\documentclass[oneside,11pt]{amsart}
\usepackage{amsmath}
\usepackage{amssymb}
\usepackage{graphicx}
\usepackage{soul}
\usepackage{caption}
\usepackage{subcaption}
\usepackage{pinlabel}
\usepackage{tikz}
\usepackage{color}
\usetikzlibrary{shapes.gates.logic.US,trees,positioning,arrows}
 \usepackage{float}

\newtheorem{thm}{Theorem}[section]
\newtheorem{prop}[thm]{Proposition}
\newtheorem{lem}[thm]{Lemma}

\theoremstyle{definition}
\newtheorem{defn}[thm]{Definition}
\newtheorem{rem}[thm]{Remark}

\newtheorem{exmp}[thm]{Example}

\newtheorem{prob}[thm]{Problem}
\newtheorem{cons}[thm]{Construction}

\newtheorem*{sassu}{Standing Assumptions}

\renewcommand{\bar}[1]{\overline{#1}}

\newcommand{\set}[2]{\{\,{#1} \mid {#2} \,\}}

\newcommand{\field}[1]{\mathbb{#1}}
\newcommand{\Z}{\field{Z}}

\newcommand{\R}{\field{R}}

\newcommand{\PP}{\field{P}}

\newcommand{\EE}{\field{E}}
\newcommand{\NN}{\field{N}}

\newcommand{\JJ}{\field{J}}
\newcommand{\ZZ}{\field{Z}}

\DeclareMathOperator{\CAT}{CAT}

\DeclareMathOperator{\Stab}{Stab}

\DeclareMathOperator{\diam}{diam}

\usepackage{ifthen}

\newcommand{\showcomments}{yes}

\newsavebox{\commentbox}
{\ifthenelse{\equal{\showcomments}{yes}}%
{\footnotemark
    \begin{lrbox}{\commentbox}
    \begin{minipage}[t]{1.25in}\raggedright\sffamily\upshape\tiny
    \footnotemark[\arabic{footnote}]}%
{\begin{lrbox}{\commentbox}}}%
{\ifthenelse{\equal{\showcomments}{yes}}%
{\end{minipage}\end{lrbox}\marginpar{\usebox{\commentbox}}}%
{\end{lrbox}}}

\begin{document}

\title[\tiny The relative hyperbolicity and manifold structure of certain RACGs]{\footnotesize On the relative hyperbolicity and manifold structure of certain right-angled Coxeter groups}

\author{Matthew Haulmark}
\address{Department of Mathematics\\
1326 Stevenson Center\\
Vanderbilt University\\
Nashville, TN 37240 USA}
\email{m.haulmark@vanderbilt.edu}

\author{Hoang Thanh Nguyen}
\address{Department of Mathematical Sciences\\
University of Wisconsin-Milwaukee\\
P.O. Box 413\\
Milwaukee, WI 53201\\
USA}
\email{nguyen36@uwm.edu}

\author{Hung Cong Tran}
\address{Department of Mathematics\\
The University of Georgia\\
1023 D. W. Brooks Drive\\
Athens, GA 30605\\
USA}
\email{hung.tran@uga.edu}

\date{\today}

\begin{abstract}
In this article, we study the manifold structure and the relatively hyperbolic structure of right-angled Coxeter groups with planar nerves. We then apply our results to the quasi-isometry problem for this class of right-angled Coxeter groups.
\end{abstract}

\subjclass[2000]{%
20F67, 
20F65} 
\maketitle

\section{Introduction}
For each finite simplicial graph $\Gamma$ the associated \emph{right-angled Coxeter group} $G_\Gamma$ has generating set $S$ equal to the vertices of $\Gamma$, relations $s^2=1$ for each $s$ in $S$ and relations $st = ts$ whenever $s$ and $t$ are adjacent vertices. The graph $\Gamma$ is the \emph{defining graph} of right-angled Coxeter group $G_\Gamma$ and its flag complex $\Delta=\Delta(\Gamma)$ is the \emph{defining nerve} of the group. Therefore, sometimes we also denote the right-angled Coxeter group $G_\Gamma$ by $G_\Delta$. In geometric group theory, groups acting on $\CAT(0)$ cube complexes are fundamental objects and right-angled Coxeter groups provide a rich source of such groups. The coarse geometry of right-angled Coxeter groups has been studied by Caprace \cite{MR2665193, MR3450952}, Dani-Thomas \cite{MR3314816, DATA}, Dani-Stark-Thomas \cite{DST}, Behrstock-Hagen-Sisto \cite{MR3623669}, Levcovitz \cite{IL} and others. In this paper, we will study the boundary of relatively hyperbolic right-angled Coxeter groups and the geometry of right-angled Coxeter groups with planar nerves.

In this paper, we study the geometry of right-angled Coxeter groups of planar defining nerves. We will focus our attention on one-ended non-hyperbolic groups. We will also exclude the virtually $\Z^2$ groups. Thus, we assume the nerves of our groups satisfy the following conditions.

\begin{sassu}The planar flag complex $\Delta\subset \field{S}^2$:
\begin{enumerate}
    \item is connected with no separating vertices and no separating edges ($G_\Delta$ is one-ended);
    \item contains at least one induced $4$-cycle ($G_\Delta$ is not hyperbolic);
    \item is not a $4$-cycle and not a cone of a $4$-cycle ($G_\Delta$ is not virtually $\Z^2$).
\end{enumerate}
\end{sassu}

\subsection{Manifold group structure and relatively hyperbolic structure}
Davis-Okun~\cite{MR1812434} and Droms~\cite{MR1974626} proved that the one-ended right-angled Coxeter group $G_\Delta$ is virtually the fundamental group of a $3$--manifold $M$ if the defining nerve $\Delta$ is planar. So, we can learn a lot about the geometry of $G_\Delta$ if we know the manifold type of $M$. This is the main motivation for proving the following:

\begin{thm}
\label{mani}
Let $\Delta\subset \field{S}^2$ be a flag complex satisfying Standing Assumptions. Then the right-angled Coxeter group $G_\Delta$ is virtually the fundamental group of a 3-manifold $M$ with empty or toroidal boundary if and only if $\Delta=\field{S}^2$ or the boundary of each region in $\field{S}^2-\Delta$ is a 4-cycle. Moreover, in the case $G_\Delta$ is virtually the fundamental group of a 3-manifold $M$ with empty or toroidal boundary, there are four mutually exclusive cases:
\begin{enumerate}
    \item
    \label{item:seifert}If $\Delta$ is a suspension of some $n$-cycle ($n\geq 4$) or some broken line (i.e a finite disjoint union of vertices and finite trees with vertex degrees $1$ or $2$), then $M$ is a Seifert manifold;
    \item
    \label{item: graphmld} If the $1$-skeleton of $\Delta$ is $\mathcal{CFS}$ (see Definition~\ref{defn:CFS}) and does not satisfy (1), then $M$ is a graph manifold; 
    \item
    \label{item:hyper}
    If the $1$-skeleton of $\Delta$ has no separating induced $4$-cycle and is not $\mathcal{CFS}$, then $M$ is a hyperbolic manifold with boundary; 
    \item 
    \label{item:mixed}
   If the 1-skeleton of $\Delta$ contains a separating induced $4$-cycle and is not $\mathcal{CFS}$, then $M$ is a mixed manifold.
\end{enumerate}

\end{thm}

 In Theorem~\ref{mani} we characterize the associated manifold $M$ of $G_\Delta$ with empty or toroidal boundary using the work in \cite{KSD} on the Euler characteristic of torsion free finite index subgroups of right-angled Coxeter groups. For further classification of such a manifold $M$ we investigate the relatively hyperbolic structure of $G_\Delta$. As a consequence, we obtain a key result on all possible divergence of $G_\Delta$.  

\begin{thm}
\label{sosonice}
Let $\Gamma$ be a graph whose flag complex $\Delta$ is planar. There is a collection $\JJ$ of $\mathcal{CFS}$ subgraphs of $\Gamma$ such that the right-angled Coxeter group $G_\Gamma$ is relatively hyperbolic with respect to the collection $\PP=\set{G_J}{J\in \JJ}$. In particular, if $G_\Gamma$ is one-ended, then the divergence of $G_\Gamma$ is linear, quadratic, or exponential.
\end{thm}

For the proof of Theorem~\ref{sosonice} we carefully investigate the tree structure of the defining nerve $\Delta$ and then follow an almost identical strategy to the proof of Theorem 1.6 in \cite{NT}. The result concerning the divergence of the group follows directly from \cite[Theorem 7.4] {IL} and \cite[Theorem 1.3]{Sisto}.

We also characterize right-angled Coxeter groups with planar nerves such that they have non-trivial minimal peripheral structures with the Sierpinski carpet as their Bowditch boundaries. This result will contribute to the study of the coarse geometry of our groups and we will discuss it in the next section. 

\begin{thm}
\label{SC-carpet}
Let $\Delta\subset \field{S}^2$ be a planar flag complex satisfying Standing Assumptions. Then $G_\Delta$ has a non-trivial minimal peripheral structure with Bowditch boundary the Sierpinski carpet if and only if all following conditions hold:
\begin{enumerate}
    \item The boundary of some region of $\field{S}^2-\Delta$ is an $n$-cycle with $n\geq 5$;
    \item The $1$-skeleton of $\Delta$ has no separating induced $4$-cycle, no cut pair, and no separating induced path of length $2$.
\end{enumerate}
\end{thm}

\subsection{Coarse geometry}

Theorem~\ref{mani} divides right-angled Coxeter groups with nerves satisfying Standing Assumptions into two main types: the ones which are virtually the fundamental group of a $3$--manifold with empty or toroidal boundary which we call \emph{type A}, and the ones which are not virtually the fundamental group of such a $3$--manifold which we will call \emph{type B}. Behrstock and Neumann summarize the rigidity results of $3$--manifolds with empty or toroidal boundary from many authors (see \cite{MR623534}, \cite{MR1036000}, \cite{MR2402598}, \cite{MR1440310}, \cite{MR1840770}, \cite{MR1383215}, and \cite{MR1898396}) in Theorem A of ~\cite{MR2376814}. Thus, a right-angled Coxeter group of type A and a right-angled Coxeter group of type B are never quasi-isometric by Theorem A of \cite{MR2376814}.

Theorem~\ref{mani} further divides right-angled Coxeter groups of type A into $4$ subtypes (see Figure~\ref{A}). It is well-known that the fundamental groups of the associated manifolds in the four different subtypes are not quasi-isometric (see Theorem~5.4 in \cite{KaLeeb95}), so 
two right-angled Coxeter groups of type A are not quasi-isometric if they are of different subtypes. Thus, we must study each subtype to understand the quasi-isometry classification of groups of type A.

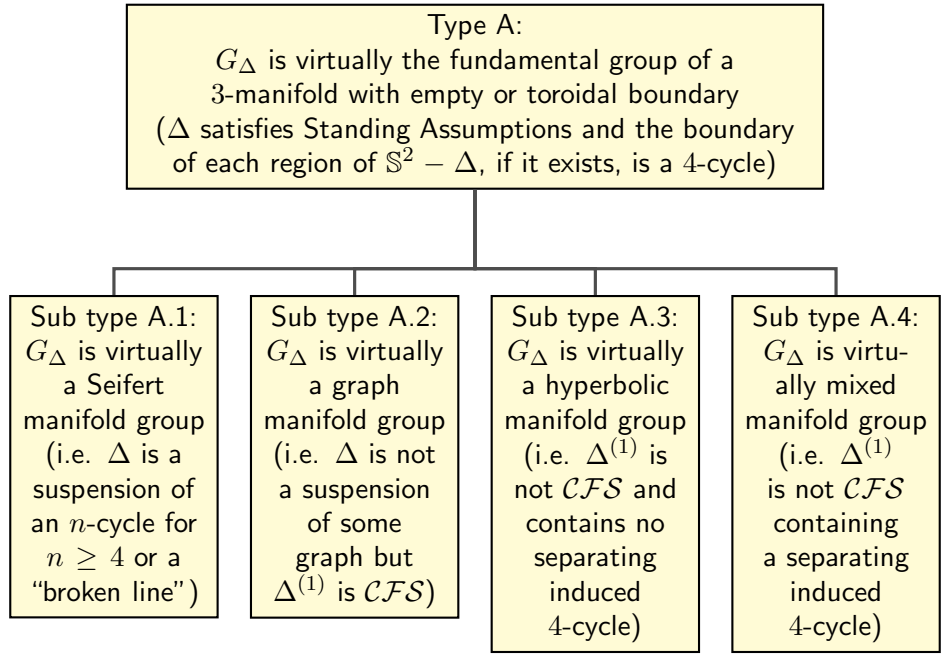
\begin{figure}
\centering
\begin{tikzpicture}[
    and/.style={and gate US,thick,draw,fill=red!60,rotate=90,
		anchor=east,xshift=-1mm},
    or/.style={or gate US,thick,draw,fill=blue!60,rotate=90,
		anchor=east,xshift=-1mm},
    be/.style={circle,thick,draw,fill=green!60,anchor=north,
		minimum width=0.7cm},
    tr/.style={buffer gate US,thick,draw,fill=purple!60,rotate=90,
		anchor=east,minimum width=0.8cm},
    label distance=3mm,
    every label/.style={blue},
    event/.style={rectangle,thick,draw,fill=yellow!20,text width=9cm,
		text centered,font=\sffamily,anchor=north},
	event1/.style={rectangle,thick,draw,fill=yellow!20,text width=2.5cm,
		text centered,font=\sffamily,anchor=north},	
    event2/.style={rectangle,thick,draw,fill=yellow!20,text width=1.5cm,
		text centered,font=\sffamily,anchor=north},	
    edge from parent/.style={very thick,draw=black!70},
    edge from parent path={(\tikzparentnode.south) -- ++(0,-1.05cm)
			-| (\tikzchildnode.north)},
    level 1/.style={sibling distance=3.2cm,level distance=1.4cm,
			growth parent anchor=south,nodes=event},
    level 2/.style={sibling distance=2.5cm},
    level 3/.style={sibling distance=6cm},
    level 4/.style={sibling distance=3cm}
    ]
    \node (g1)[event] {Type A:\\$G_\Delta$ is virtually the fundamental group of a $3$-manifold with empty or toroidal boundary \\($\Delta$ satisfies Standing Assumptions and the boundary of each region of $\field{S}^2-\Delta$, if it exists, is a $4$-cycle)}
	     	child {node (g3)[event1] {Sub type A.1:\\$G_\Delta$ is virtually a Seifert manifold group\\(i.e. $\Delta$ is a suspension of an $n$-cycle for $n\geq 4$ or a ``broken line'')}}
	     	child {node (g3)[event1] {Sub type A.2:\\ $G_\Delta$ is virtually a graph manifold group\\(i.e. $\Delta$ is not a suspension of some graph but $\Delta^{(1)}$ is $\mathcal{CFS}$)}}
	     	child {node (g3)[event1] {Sub type A.3:\\$G_\Delta$ is virtually a hyperbolic manifold group\\(i.e. $\Delta^{(1)}$ is not $\mathcal{CFS}$ and contains no separating induced $4$-cycle)}}
	     	child {node (b1)[event1] {Sub type A.4:\\$G_\Delta$ is virtually a mixed manifold group\\(i.e. $\Delta^{(1)}$ is not $\mathcal{CFS}$ containing a separating induced $4$-cycle)}
	     			};
	\end{tikzpicture}
	\caption{There are four subtypes of type A, and groups of different subtypes are not quasi-isometric.}
	\label{A}
\end{figure}

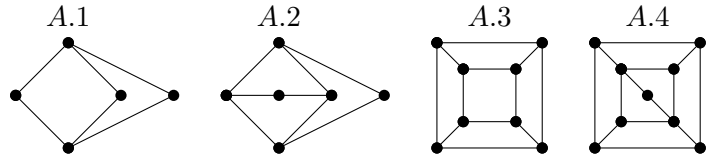
\begin{figure}
\begin{tikzpicture}[scale=0.7]

\draw (0,0) node[circle,fill,inner sep=1.5pt, color=black](1){} -- (1,1) node[circle,fill,inner sep=1.5pt, color=black](1){}-- (2,0) node[circle,fill,inner sep=1.5pt, color=black](1){}-- (1,-1) node[circle,fill,inner sep=1.5pt, color=black](1){} -- (0,0) node[circle,fill,inner sep=1.5pt, color=black](1){}; \draw (1,1) node[circle,fill,inner sep=1.5pt, color=black](1){} -- (3,0) node[circle,fill,inner sep=1.5pt, color=black](1){}-- (1,-1) node[circle,fill,inner sep=1.5pt, color=black](1){};
\node at (1,1.5) {$A.1$};

\draw (0+4,0) node[circle,fill,inner sep=1.5pt, color=black](1){} -- (1+4,1) node[circle,fill,inner sep=1.5pt, color=black](1){}-- (2+4,0) node[circle,fill,inner sep=1.5pt, color=black](1){}-- (1+4,-1) node[circle,fill,inner sep=1.5pt, color=black](1){} -- (0+4,0) node[circle,fill,inner sep=1.5pt, color=black](1){}; \draw (1+4,1) node[circle,fill,inner sep=1.5pt, color=black](1){} -- (3+4,0) node[circle,fill,inner sep=1.5pt, color=black](1){}-- (1+4,-1) node[circle,fill,inner sep=1.5pt, color=black](1){};\draw (0+4,0) node[circle,fill,inner sep=1.5pt, color=black](1){} -- (1+4,0) node[circle,fill,inner sep=1.5pt, color=black](1){}-- (2+4,0) node[circle,fill,inner sep=1.5pt, color=black](1){}; 
\node at (1+4,1.5) {$A.2$};

\draw (8,1) node[circle,fill,inner sep=1.5pt, color=black](1){} -- (10,1) node[circle,fill,inner sep=1.5pt, color=black](1){}-- (10,-1) node[circle,fill,inner sep=1.5pt, color=black](1){}-- (8,-1) node[circle,fill,inner sep=1.5pt, color=black](1){} -- (8,1) node[circle,fill,inner sep=1.5pt, color=black](1){}; \draw (8.5,0.5) node[circle,fill,inner sep=1.5pt, color=black](1){} -- (9.5,0.5) node[circle,fill,inner sep=1.5pt, color=black](1){}-- (9.5,-0.5) node[circle,fill,inner sep=1.5pt, color=black](1){}-- (8.5,-0.5) node[circle,fill,inner sep=1.5pt, color=black](1){} -- (8.5,0.5) node[circle,fill,inner sep=1.5pt, color=black](1){}; \draw (8,1) node[circle,fill,inner sep=1.5pt, color=black](1){} -- (8.5,0.5) node[circle,fill,inner sep=1.5pt, color=black](1){};\draw (10,1) node[circle,fill,inner sep=1.5pt, color=black](1){} -- (9.5,0.5) node[circle,fill,inner sep=1.5pt, color=black](1){};\draw (10,-1) node[circle,fill,inner sep=1.5pt, color=black](1){} -- (9.5,-0.5) node[circle,fill,inner sep=1.5pt, color=black](1){};\draw (8,-1) node[circle,fill,inner sep=1.5pt, color=black](1){} -- (8.5,-0.5) node[circle,fill,inner sep=1.5pt, color=black](1){};
\node at (9,1.5) {$A.3$};

\draw (8+3,1) node[circle,fill,inner sep=1.5pt, color=black](1){} -- (10+3,1) node[circle,fill,inner sep=1.5pt, color=black](1){}-- (10+3,-1) node[circle,fill,inner sep=1.5pt, color=black](1){}-- (8+3,-1) node[circle,fill,inner sep=1.5pt, color=black](1){} -- (8+3,1) node[circle,fill,inner sep=1.5pt, color=black](1){}; \draw (8.5+3,0.5) node[circle,fill,inner sep=1.5pt, color=black](1){} -- (9.5+3,0.5) node[circle,fill,inner sep=1.5pt, color=black](1){}-- (9.5+3,-0.5) node[circle,fill,inner sep=1.5pt, color=black](1){}-- (8.5+3,-0.5) node[circle,fill,inner sep=1.5pt, color=black](1){} -- (8.5+3,0.5) node[circle,fill,inner sep=1.5pt, color=black](1){}; \draw (8+3,1) node[circle,fill,inner sep=1.5pt, color=black](1){} -- (8.5+3,0.5) node[circle,fill,inner sep=1.5pt, color=black](1){};\draw (10+3,1) node[circle,fill,inner sep=1.5pt, color=black](1){} -- (9.5+3,0.5) node[circle,fill,inner sep=1.5pt, color=black](1){};\draw (10+3,-1) node[circle,fill,inner sep=1.5pt, color=black](1){} -- (9.5+3,-0.5) node[circle,fill,inner sep=1.5pt, color=black](1){};\draw (8+3,-1) node[circle,fill,inner sep=1.5pt, color=black](1){} -- (8.5+3,-0.5) node[circle,fill,inner sep=1.5pt, color=black](1){}; \draw (8.5+3,0.5) node[circle,fill,inner sep=1.5pt, color=black](1){} -- (9+3,0) node[circle,fill,inner sep=1.5pt, color=black](1){}-- (9.5+3,-0.5) node[circle,fill,inner sep=1.5pt, color=black](1){};
\node at (9+3,1.5) {$A.4$};

\end{tikzpicture}

\caption{Examples of type $A$ nerves}
\label{A'}
\end{figure}

The collection of groups of subtype A.1 contains exactly three quasi-isometry equivalence classes. The first quasi-isometry equivalence class of groups of type A.1 contains only one right-angled Coxeter group with nerve a suspension of a $4$--cycle. This right-angled Coxeter group is virtually $\Z^3$. The second quasi-isometry equivalence class consists of right-angled Coxeter groups whose nerve is a suspension of some $n$--cycle ($n\geq 5$). Each right-angled Coxeter groups in this second class is virtually the fundamental group of a closed Seifert manifold. The last quasi-isometry equivalence class consists of right-angled Coxeter groups with nerve a suspension of some broken line. Each right-angled Coxeter group in this class is virtually the fundamental group of a Seifert manifold with non-empty boundary.

The quasi-isometry classification of groups of subtype A.2 is much more complicated. We give a complete quasi-isometry classification of groups of subtype $A.2$. 

\begin{thm}
\label{thm:graphmanifold}
Let $\Delta\subset \field{S}^2$ and $\Delta'\subset \field{S}^2$ be two planar flag complexes satisfying Standing Assumptions such that their $1$--skeletons are $\mathcal{CFS}$ and not join graphs. Let $T_r$ and $T_r'$ be the visual decomposition trees of $\Delta$ and $\Delta'$ respectively. Then two groups $G_{\Delta}$ and $G_{\Delta'}$ are quasi-isometric if and only if $T_r$ and $T_r'$ are bisimilar.
\end{thm}

The above theorem strengthens Theorem 1.1 of \cite{NT} by removing the condition ``triangle free'' from the hypothesis of the theorem. We remark that the visual decomposition trees of the defining nerves in Theorems~\ref{thm:graphmanifold} are colored trees whose vertices are colored black and white, and they are constructed in Sections~\ref{tree1} and \ref{tree2}. We also refer the reader to Section~\ref{bisimilarity} for the notion of bisimilarity among two-colored graphs.

We also classify groups of subtype A.2 which are quasi-isometric to right-angled Artin groups. The following theorem is an extension of Theorem 1.2 of \cite{NT}.

\begin{thm}
\label{thm:racgraag}
Let $\Delta\subset \field{S}^2$ be a planar flag complex satisfying Standing Assumptions with the $1$--skeleton a non-join $\mathcal{CFS}$. Let $T_r$ be a visual decomposition tree of $\Delta$. Then the following are equivalent:
\begin{enumerate}
\item The right-angled Coxeter group $G_\Delta$ is quasi-isometric to a right-angled Artin group.
\item The right-angled Coxeter group $G_\Delta$ is quasi-isometric to the right-angled Artin group of a tree of diameter at least 3.
\item The right-angled Coxeter group $G_\Delta$ is quasi-isometric to the right-angled Artin group of a tree of diameter exactly 3.
\item All vertices of the tree $T_r$ are black.
\end{enumerate}
\end{thm}

All right-angled Coxeter groups of subtype A.3 are virtually fundamental groups of hyperbolic $3$--manifolds. So far, the authors do not know a complete quasi-isometric classification for the groups in this subtype. By Schwartz Rigidity Theorem \cite{Schwartz95} (see also, for example, Theorem~24.1 \cite{DK18}) the fundamental groups of two hyperbolic $3$--manifolds are quasi-isometric if and only if they are commensurable. Therefore, we turn the quasi-isometry classification problem for right-angled Coxeter groups of type A.3 into a commensurability problem.

\begin{prob}
Classify all right-angled Coxeter groups which are virtually the fundamental group of a hyperbolic $3$--manifold up to commensurability. 
\end{prob}

All right-angled Coxeter groups of subtype A.4 are virtually the fundamental groups of mixed $3$--manifolds. Therefore, the last conclusion of Theorem~\ref{mani} potentially allows us to use the work of Kapovich-Leeb~\cite{MR1440310,KL98} to understand the geometry of right-angled Coxeter groups of subtype A.4. The complete quasi-isometry classification of fundamental groups of mixed manifolds remains an open question. As a consequence, the authors do not know the complete quasi-isometry classification of groups of subtype A.4.

We note that all right-angled Coxeter groups of type B are non-trivially relatively hyperbolic (see Proposition~\ref{pppp1}). So, we divide these groups into three different subtypes based on their Bowditch boundary with respect to the minimal peripheral structure and whether they split over $2$--ended subgroups (see Figure~\ref{B}).

\begin{figure}
\centering

\begin{tikzpicture}[
    and/.style={and gate US,thick,draw,fill=red!60,rotate=90,
		anchor=east,xshift=-1mm},
    or/.style={or gate US,thick,draw,fill=blue!60,rotate=90,
		anchor=east,xshift=-1mm},
    be/.style={circle,thick,draw,fill=green!60,anchor=north,
		minimum width=0.7cm},
    tr/.style={buffer gate US,thick,draw,fill=purple!60,rotate=90,
		anchor=east,minimum width=0.8cm},
    label distance=3mm,
    every label/.style={blue},
    event/.style={rectangle,thick,draw,fill=yellow!20,text width=10cm,
		text centered,font=\sffamily,anchor=north},
	event1/.style={rectangle,thick,draw,fill=yellow!20,text width=3cm,
		text centered,font=\sffamily,anchor=north},	
    event2/.style={rectangle,thick,draw,fill=yellow!20,text width=1.5cm,
		text centered,font=\sffamily,anchor=north},	
    edge from parent/.style={very thick,draw=black!70},
    edge from parent path={(\tikzparentnode.south) -- ++(0,-1.05cm)
			-| (\tikzchildnode.north)},
    level 1/.style={sibling distance=4cm,level distance=1.4cm,
			growth parent anchor=south,nodes=event},
    level 2/.style={sibling distance=2.5cm},
    level 3/.style={sibling distance=6cm},
    level 4/.style={sibling distance=3cm}
    ]
    \node (g1)[event] {Type B:\\$G_\Delta$ is not virtually the fundamental group of a $3$-manifold with empty or toroidal boundary \\(i.e. $\Delta$ satisfies Standing Assumptions and the boundary of some region of $\field{S}^2-\Delta$ is not a $4$-cycle)}
	     	child {node (g3)[event1] {Sub type B.1:\\$G_\Delta$ does not split over virtually $\ZZ$ or $\ZZ^2$ groups (i.e. $\Delta^{(1)}$ has no cut pair, no separating induced path of lengh $2$, and no separating induced $4$-cycle)}}
	     	child {node (g3)[event1] {Sub type B.2:\\$G_\Delta$ splits over a virtually $\ZZ$ subgroup (i.e. $\Delta^{(1)}$ has a cut pair or a separating induced path of lengh $2$)}}
	     	child {node (g3)[event1] {Sub type B.3:\\$G_\Delta$ splits over a virtually $\ZZ^2$ subgroup, but does not split over a virtually $\ZZ$ subgroup (i.e. $\Delta^{(1)}$ has a separating $4$-cycle but no cut pair and no separating induced path of length $2$)}};
	\end{tikzpicture}
	\caption{Three different subtypes of type~B. Groups in the different subtypes B.1, B.2, and B.3 are not quasi-isometric.}
	\label{B}
\end{figure}
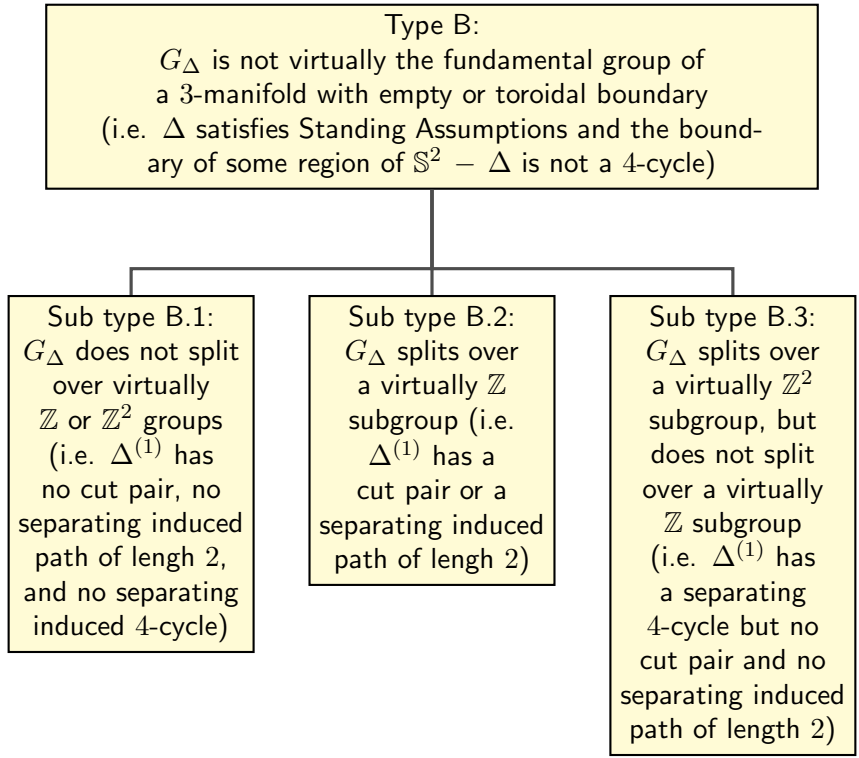

Right-angled Coxeter groups of subtype B.1 all have minimal peripheral structure whose Bowditch boundary is a Sierpinski carpet (see Theorem~\ref{SC-carpet}). Meanwhile, the Bowditch boundary of right-angled Coxeter groups of subtypes B.2 and B.3 with respect to a minimal peripheral structure must have a cut point or a non-parabolic cut pair by Theorems~\ref{intro2} and \ref{intro3}. Therefore, a right-angled Coxeter group of subtype B.1 is not quasi-isometric to one of subtype B.2 and subtype B.3. 

A right-angled Coxeter group of subtype B.2 always splits over a $2$--ended subgroup while a right-angled Coxeter group of subtype B.3 does not. Therefore, they are never quasi-isometric by the work of Papasoglu \cite{MR2153400}. (Papasoglu showed that among $1$-ended finitely presented groups that are not commensurable to surface groups, having a splitting over a 2-ended subgroup is a quasi-isometry invariant.) We hope to use the work of Cashen-Martin~\cite{MR3604915} to understand the coarse geometry of right-angled Coxeter groups of subtype B.2. 

It seems hard to obtain the complete quasi-isometry classification of groups of type B. However, the quasi-isometry classification of the ones of subtype B.1 seems reasonable. 
\begin{prob}
Classify up to quasi-isometry all relatively hyperbolic right-angled Coxeter groups whose Bowditch boundary is a Sierpinski carpet with respect to some minimal peripheral structure.
\end{prob}

\subsection{Outline of the paper}

In Section \ref{prelim}, we review some concepts in geometric group theory: $\CAT(0)$ spaces, $\delta$--hyperbolic spaces, $\CAT(0)$ spaces with isolated flats, $3$--manifold groups, relatively hyperbolic groups, $\CAT(0)$ boundaries, Gromov boundaries, Bowditch boundaries, and peripheral splitting of relatively hyperbolic groups. We also recall some necessary results related to these concepts. In Section \ref{visual splittings and Bowditch boundary}, we give descriptions of cut points and non-parabolic cut pairs of Bowditch boundaries of relatively hyperbolic right-angled Coxeter groups (see Theorems \ref{intro2} and \ref{intro3}). In Section \ref{nice3}, we study the coarse geometry of right-angled Coxeter groups with planar defining nerves. We first analyze the tree structure of planar flag complexes (see Subsection~\ref{tree1}). Then we use this structure to study the relatively hyperbolic structure, group divergence, and manifold structure of right-angled Coxeter groups with planar nerves. We give the proofs of Theorem~\ref{mani}, Theorem~\ref{sosonice}, and Theorem~\ref{SC-carpet} in Subsection \ref{rhsms}. We improve the tree structure of planar flag complexes with $1$--skeletons $\mathcal{CSF}$ graphs and give a complete quasi-isometry classification of right-angled Coxeter groups which are virtually graph manifold groups in Subsection~\ref{tree2}. The proof of Theorem~\ref{thm:graphmanifold} is also given in this subsection.
	
\subsection*{Acknowledgments}
All three authors would like to thank Chris Hruska for suggestions and insights. The authors are also grateful for the insightful comments of Jason Behrstock, Ivan Levcovitz, and Mike Mihalik that have helped improve the exposition of this paper. The first author would like to thank Genevieve Walsh for helpful conversations concerning to this paper. Lastly, all three authors are particularly grateful for the referee's helpful comments and suggestions. 


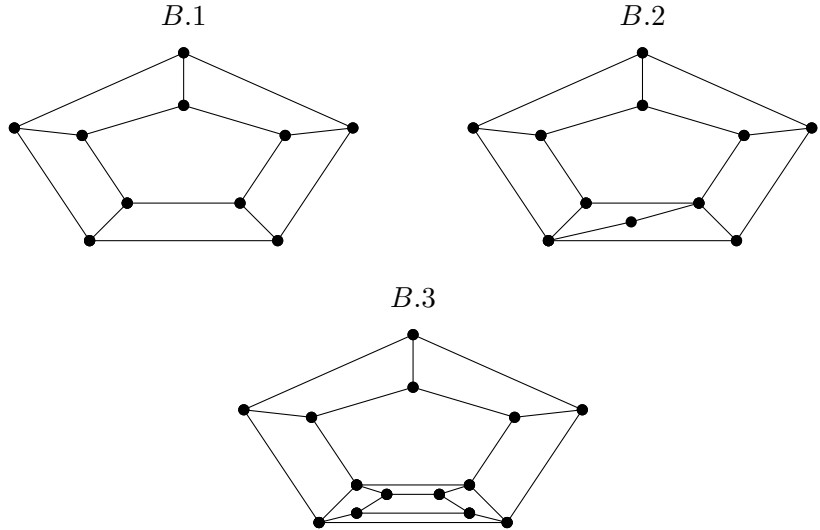
\begin{figure}
\begin{tikzpicture}[scale=0.5]

\draw (4,0) node[circle,fill,inner sep=1.5pt, color=black](1){} -- (2,3) node[circle,fill,inner sep=1.5pt, color=black](1){} -- (6.5,5) node[circle,fill,inner sep=1.5pt, color=black](1){}-- (11,3) node[circle,fill,inner sep=1.5pt, color=black](1){} -- (9,0) node[circle,fill,inner sep=1.5pt, color=black](1){} -- (4,0) node[circle,fill,inner sep=1.5pt, color=black](1){};

\draw (5,1) node[circle,fill,inner sep=1.5pt, color=black](1){} -- (3.8,2.8) node[circle,fill,inner sep=1.5pt, color=black](1){} -- (6.5,3.6) node[circle,fill,inner sep=1.5pt, color=black](1){}-- (9.2,2.8) node[circle,fill,inner sep=1.5pt, color=black](1){} -- (8,1) node[circle,fill,inner sep=1.5pt, color=black](1){} -- (5,1) node[circle,fill,inner sep=1.5pt, color=black](1){};

\draw (4,0) node[circle,fill,inner sep=1.5pt, color=black](1){} -- (5,1) node[circle,fill,inner sep=1.5pt, color=black](1){};
\draw (2,3) node[circle,fill,inner sep=1.5pt, color=black](1){} -- (3.8,2.8) node[circle,fill,inner sep=1.5pt, color=black](1){};
\draw (6.5,5) node[circle,fill,inner sep=1.5pt, color=black](1){} -- (6.5,3.6) node[circle,fill,inner sep=1.5pt, color=black](1){};
\draw (11,3) node[circle,fill,inner sep=1.5pt, color=black](1){} --  (9.2,2.8) node[circle,fill,inner sep=1.5pt, color=black](1){};
\draw (9,0) node[circle,fill,inner sep=1.5pt, color=black](1){} --  (8,1) node[circle,fill,inner sep=1.5pt, color=black](1){};

\node at (6.5,6.0) {$B.1$};

\draw (16.2,0) node[circle,fill,inner sep=1.5pt, color=black](1){} -- (14.2,3) node[circle,fill,inner sep=1.5pt, color=black](1){} -- (18.7,5) node[circle,fill,inner sep=1.5pt, color=black](1){}-- (23.2,3) node[circle,fill,inner sep=1.5pt, color=black](1){} -- (21.2,0) node[circle,fill,inner sep=1.5pt, color=black](1){} -- (16.2,0) node[circle,fill,inner sep=1.5pt, color=black](1){};

\draw (17.2,1) node[circle,fill,inner sep=1.5pt, color=black](1){} -- (16,2.8) node[circle,fill,inner sep=1.5pt, color=black](1){} -- (18.7,3.6) node[circle,fill,inner sep=1.5pt, color=black](1){}-- (21.4,2.8) node[circle,fill,inner sep=1.5pt, color=black](1){} -- (20.2,1) node[circle,fill,inner sep=1.5pt, color=black](1){} -- (17.2,1) node[circle,fill,inner sep=1.5pt, color=black](1){};

\draw (16.2,0) node[circle,fill,inner sep=1.5pt, color=black](1){} --  (17.2,1) node[circle,fill,inner sep=1.5pt, color=black](1){};
\draw (14.2,3) node[circle,fill,inner sep=1.5pt, color=black](1){} --  (16,2.8) node[circle,fill,inner sep=1.5pt, color=black](1){};
\draw (18.7,5) node[circle,fill,inner sep=1.5pt, color=black](1){} --  (18.7,3.6) node[circle,fill,inner sep=1.5pt, color=black](1){};
\draw (23.2,3) node[circle,fill,inner sep=1.5pt, color=black](1){} --  (21.4,2.8) node[circle,fill,inner sep=1.5pt, color=black](1){};
\draw (21.2,0) node[circle,fill,inner sep=1.5pt, color=black](1){} --  (20.2,1) node[circle,fill,inner sep=1.5pt, color=black](1){};
\draw (16.2,0) node[circle,fill,inner sep=1.5pt, color=black](1){} --  (18.4,0.5) node[circle,fill,inner sep=1.5pt, color=black](1){}--(20.2,1) node[circle,fill,inner sep=1.5pt, color=black](1){};


\node at (18.7,6.0) {$B.2$};

\draw (16.2+12.2-18.3,0-7.5) node[circle,fill,inner sep=1.5pt, color=black](1){} -- (14.2+12.2-18.3,3-7.5) node[circle,fill,inner sep=1.5pt, color=black](1){} -- (18.7+12.2-18.3,5-7.5) node[circle,fill,inner sep=1.5pt, color=black](1){}-- (23.2+12.2-18.3,3-7.5) node[circle,fill,inner sep=1.5pt, color=black](1){} -- (21.2+12.2-18.3,0-7.5) node[circle,fill,inner sep=1.5pt, color=black](1){} -- (16.2+12.2-18.3,0-7.5) node[circle,fill,inner sep=1.5pt, color=black](1){};

\draw (17.2+12.2-18.3,1-7.5) node[circle,fill,inner sep=1.5pt, color=black](1){} -- (16+12.2-18.3,2.8-7.5) node[circle,fill,inner sep=1.5pt, color=black](1){} -- (18.7+12.2-18.3,3.6-7.5) node[circle,fill,inner sep=1.5pt, color=black](1){}-- (21.4+12.2-18.3,2.8-7.5) node[circle,fill,inner sep=1.5pt, color=black](1){} -- (20.2+12.2-18.3,1-7.5) node[circle,fill,inner sep=1.5pt, color=black](1){} -- (17.2+12.2-18.3,1-7.5) node[circle,fill,inner sep=1.5pt, color=black](1){};

\draw (16.2+12.2-18.3,0-7.5) node[circle,fill,inner sep=1.5pt, color=black](1){} --  (17.2+12.2-18.3,1-7.5) node[circle,fill,inner sep=1.5pt, color=black](1){};
\draw (14.2+12.2-18.3,3-7.5) node[circle,fill,inner sep=1.5pt, color=black](1){} --  (16+12.2-18.3,2.8-7.5) node[circle,fill,inner sep=1.5pt, color=black](1){};
\draw (18.7+12.2-18.3,5-7.5) node[circle,fill,inner sep=1.5pt, color=black](1){} --  (18.7+12.2-18.3,3.6-7.5) node[circle,fill,inner sep=1.5pt, color=black](1){};
\draw (23.2+12.2-18.3,3-7.5) node[circle,fill,inner sep=1.5pt, color=black](1){} --  (21.4+12.2-18.3,2.8-7.5) node[circle,fill,inner sep=1.5pt, color=black](1){};
\draw (21.2+12.2-18.3,0-7.5) node[circle,fill,inner sep=1.5pt, color=black](1){} --  (20.2+12.2-18.3,1-7.5) node[circle,fill,inner sep=1.5pt, color=black](1){};

\draw (20.2+12.2-18.3,0+0.25-7.5) node[circle,fill,inner sep=1.5pt, color=black](1){} -- (17.2+12.2-18.3,0+0.25-7.5) node[circle,fill,inner sep=1.5pt, color=black](1){};
\draw (19.4+12.2-18.3,0+0.75-7.5) node[circle,fill,inner sep=1.5pt, color=black](1){} -- (18.0+12.2-18.3,0+0.75-7.5) node[circle,fill,inner sep=1.5pt, color=black](1){};
\draw (20.2+12.2-18.3,0+0.25-7.5) node[circle,fill,inner sep=1.5pt, color=black](1){} -- (19.4+12.2-18.3,0+0.75-7.5) node[circle,fill,inner sep=1.5pt, color=black](1){};
\draw (17.2+12.2-18.3,0+0.25-7.5) node[circle,fill,inner sep=1.5pt, color=black](1){} -- (18.0+12.2-18.3,0+0.75-7.5) node[circle,fill,inner sep=1.5pt, color=black](1){};
\draw (20.2+12.2-18.3,0+0.25-7.5) node[circle,fill,inner sep=1.5pt, color=black](1){} -- (21.2+12.2-18.3,0-7.5) node[circle,fill,inner sep=1.5pt, color=black](1){};
\draw (17.2+12.2-18.3,0+0.25-7.5) node[circle,fill,inner sep=1.5pt, color=black](1){} -- (16.2+12.2-18.3,0-7.5) node[circle,fill,inner sep=1.5pt, color=black](1){};
\draw (19.4+12.2-18.3,0+0.75-7.5) node[circle,fill,inner sep=1.5pt, color=black](1){} -- (20.2+12.2-18.3,1-7.5) node[circle,fill,inner sep=1.5pt, color=black](1){};
\draw (18.0+12.2-18.3,0+0.75-7.5) node[circle,fill,inner sep=1.5pt, color=black](1){} -- (17.2+12.2-18.3,1-7.5) node[circle,fill,inner sep=1.5pt, color=black](1){};



\node at (18.7+12.2-18.3,6.0-7.5) {$B.3$};

\end{tikzpicture}

\caption{Examples of type $B$ nerves}
\label{B'}
\end{figure}

\section{Preliminaries}
\label{prelim}



In this section, we review some concepts in geometric group theory: $\CAT(0)$ spaces, $\delta$--hyperbolic spaces, $\CAT(0)$ spaces with isolated flats, relatively hyperbolic groups, $\CAT(0)$ boundaries, Gromov boundaries, Bowditch boundaries, and peripheral splitting of relatively hyperbolic groups. We also use the work of Behrstock-Dru{\c{t}}u-Mosher \cite{MR2501302} and Groff \cite{MR3152211} to prove that the Bowditch boundary is a quasi-isometry invariant among relatively hyperbolic groups with minimal peripheral structures. We discuss the work of Caprace \cite{MR2665193, MR3450952}, Behrstock-Hagen-Sisto \cite{MR3623669}, Dani-Thomas \cite{MR3314816}, \cite{IL}, and \cite{Sisto} on peripheral structures of relatively hyperbolic right-angled Coxeter groups and divergence of right-angled Coxeter groups. We also discuss the work of Gersten \cite{Gersten94} and Kapovich--Leeb \cite{KL98} on divergence of $3$--manifold groups. We also mention the concept of colored graphs and the bisimilarity equivalence relation on such graphs.

\subsection{$\CAT(0)$ spaces, $\delta$--hyperbolic spaces, and relatively hyperbolic groups}
We first discuss $\CAT(0)$ spaces, $\delta$--hyperbolic spaces, Gromov boundaries, and $\CAT(0)$ boundaries. We refer the reader to the book \cite{MR1744486} for more details.

\begin{defn}
We say that a geodesic triangle $\Delta$ in a geodesic space $X$ satisfies the \emph{$CAT(0)$ inequality} if $d(x,y) \leq d(\bar{x},\bar{y})$ for all points $x$, $y$ on the edges of $\Delta$ and the corresponding points $\bar{x}, \bar{y}$ on the edges of the comparison triangle $\bar{\Delta}$ in Euclidean space $\EE^2$.
\end{defn}

\begin{defn}
A geodesic space $X$ is said to be a \emph{$CAT(0)$ space} if every triangle in $X$ satisfies the $\CAT(0)$ inequality.

If $X$ is a $\CAT(0)$ space, then the \emph{$CAT(0)$ boundary} of $X$, denoted $\partial X$, is defined to be the set of all equivalence classes of geodesic rays in $X$, where two rays $c$ and $c'$ are equivalent if the Hausdorff distance between them is finite.


We note that for any $x \in X$ and $\xi \in \partial X$ there is a unique geodesic ray $\alpha_{x,\xi}: [0,\infty) \to X$ with $\alpha_{x,\xi}(0)=x$ and $\alpha_{x,\xi}(\infty)= \xi$. The CAT(0) boundary has a natural topology with a basis of local neighborhoods (not necessarily open) of any point $\xi\in \partial X$ given by the sets $U(\xi;x, R, \epsilon) = \set{\xi'\in \partial X}{d\bigl(\alpha_{x,\xi}(R),\alpha_{x,\xi'}(R)\leq \epsilon\bigr)}$, where $x \in X$, $\xi \in \partial X$, $R >0$ and $\epsilon >0$.
\end{defn}

\begin{defn}
A geodesic metric space $(X, d)$ is \emph{$\delta$--hyperbolic} if every geodesic triangle with vertices in $X$ is \emph{$\delta$--thin} in the sense that each side lies in the $\delta$--neighborhood of the union of other sides. If $X$ is a $\delta$--hyperbolic space, then we could build the Gromov boundary of $X$, denoted $\partial X$, in the same way as for a $\CAT(0)$ space. That is, the Gromov boundary of $X$ is defined to be the set of all equivalence classes of geodesic rays in $X$, where two rays $c$ and $c'$ are equivalent if the Hausdorff distance between them is finite. However, the topology on it is slightly different from the topology on the boundary of a $\CAT(0)$ space (see for example \cite[Section III.3]{MR1744486} for details).

\end{defn}

We now review relatively hyperbolic groups and related concepts.

\begin{defn}[Combinatorial horoball \cite{MR2448064}]
Let $T$ be any graph with the vertex set $V$. We define the \emph{combinatorial horoball} based at $T$, $\mathcal{H}$($=\mathcal{H}(T)$) to be the following graph:
\begin{enumerate}
\item $\mathcal{H}^{(0)}= V\times \{\{0\}\cup \NN\}$.
\item $\mathcal{H}^{(1)} = \{((t, n), (t, n + 1))\}\cup \set{((t_1, n), (t_2, n))}{d_T(t_1,t_2)\leq 2^n}$. We call edges of the first set \emph{vertical} and of the second \emph{horizontal}.
\end{enumerate}
\end{defn}

\begin{rem}
In \cite{MR2448064}, the combinatorial horoball is described as a 2-complex, but we only require the 1-skeleton for the horoball in this paper.
\end{rem}

\begin{defn}[\cite{MR2448064}]
Let $\mathcal{H}$ be the horoball based at some graph $T$. Let $D\!: \mathcal{H} \to [0,\infty)$ be defined by extending the map on vertices $(t,n)\to n$ linearly across edges. We call $D$ the \emph{depth function} for $\mathcal{H}$ and refer to vertices $v$ with $D(v)=n$ as \emph{vertices of depth $n$} or \emph{depth $n$ vertices}.

Because $T \times \{0\}$ is homeomorphic to $T$, we identify $T$ with $D^{-1}(0)$.
\end{defn}

\begin{defn}[Cusped space \cite{MR2448064}]
\label{cspaces}
Let $G$ be a finitely generated group and $\PP$ a finite collection of finitely generated subgroups of $G$. Let $S$ be a finite generating set of $G$ such that $S\cap P$ generates $P$ for each $P\in\PP$. For each left coset $gP$ of subgroup $P\in\PP$ let $\mathcal{H}(gP)$ be the horoball based at a copy of the subgraph $T_{gP}$ with vertex set $gP$ of the Cayley graph $\Gamma(G,S)$. The \emph{cusped space} $X(G,\PP,S)$ is the union of $\Gamma(G,S)$ with $\mathcal{H}(gP)$ for every left coset of $P\in \PP$, identifying the subgraph $T_{gP}$ with the depth 0 subset of $\mathcal{H}(gP)$. We suppress mention of $S$ and $\PP$ when they are clear from the context.
\end{defn}

\begin{defn}[Relatively hyperbolic group \cite{MR2448064}]
\label{minimal}
Let $G$ be a finitely generated group and $\PP$ a finite collection of finitely generated subgroups of $G$. Let $S$ be a finite generating set of $G$ such that $S\cap P$ generates $P$ for each $P\in\PP$. If the cusped space $X(G,\PP,S)$ is $\delta$--hyperbolic then we say that $G$ is \emph{hyperbolic relative to} $\PP$ or that $(G,\PP)$ is \emph{relatively hyperbolic}. Collection $\PP$ is a \emph{peripheral structure}, each group $P \in \PP$ is a \emph{peripheral subgroup} and its left cosets are \emph{peripheral left cosets}. The peripheral structure $\PP$ is \emph{minimal} if for any other peripheral structure $\field{Q}$ on $G$, each $P\in\PP$ is conjugate into some $Q\in\field{Q}$.
\end{defn}

\begin{rem}
Replacing $S$ for some other finite generating set $S'$ may change the value of $\delta$, but does not affect the hyperbolicity of the cusped space for some $\delta'$ (see \cite{MR2448064}). As a consequence, the concept of relatively hyperbolic group does not depend on the choice of finite generating set.

We say that a finitely generated group is \emph{non-trivially relatively hyperbolic} if it is relatively hyperbolic with respect to some collection of proper subgroups.

\end{rem}

\subsection{The Bowditch boundary}

We now discuss the Bowditch boundary of a relatively hyperbolic group and prove that it is a quasi-isometry invariant when the peripheral structure is minimal. We also recall peripheral splitting of relatively hyperbolic groups.

\begin{defn}[Bowditch boundary \cite{MR2922380}]
\label{defn001}
Let $(G,\PP)$ be a finitely generated relatively hyperbolic group. Let $S$ be a finite generating set of $G$ such that $S\cap P$ generates $P$ for each $P\in\PP$. The \emph{Bowditch boundary}, denoted $\partial(G,\PP)$, is the Gromov boundary of the associated cusped space, $X(G,\PP,S)$.
\end{defn}

\begin{rem}
There is a natural topological action of $G$ on the Bowditch boundary $\partial(G,\PP)$ that satisfies certain properties (see \cite{MR2922380}).

Bowditch has shown that the Bowditch boundary does not depend on the choice of finite generating set (see \cite{MR2922380}). More precisely, if $S$ and $T$ are finite generating sets for $G$ as in the above definition, then the Gromov boundaries of the cusped spaces $X(G,\PP,S)$ and $X(G,\PP,T)$ are $G$--equivariantly homeomorphic (see \cite{MR2922380}).

For each peripheral left coset $gP$ the limit set of the associated horoball $\mathcal{H}(gP)$ consists of a single point in $\partial(G,\PP)$, called the \emph{parabolic point labelled by $gP$}, denoted $v_{gP}$. The stabilizer of the point $v_{gP}$ is the subgroup $gPg^{-1}$. We call each infinite subgroup of $gPg^{-1}$ a \emph{parabolic subgroup} and subgroup $gPg^{-1}$ a \emph{maximal parabolic subgroup}.
\end{rem}

The homeomorphism type of the Bowditch boundary was already known to be a quasi-isometry invariant of a relatively hyperbolic group with minimal peripheral structure (see Groff \cite{MR3152211}). However, we combine the work of Groff \cite{MR3152211} and Behrstock-Dru{\c{t}}u-Mosher \cite{MR2501302} to elaborate on the homeomorphism between Bowditch boundaries induced by a quasi-isometry between two relatively hyperbolic groups with minimal peripheral structures.

\begin{thm}
\label{main}
Let $(G_1,\PP_1)$ and $(G_2,\PP_2)$ be finitely generated non-trivially relatively hyperbolic groups with minimal peripheral structures. If $G_1$ and $G_2$ are quasi-isometric, then there is a homeomorphism $f$ from $\partial(G_1,\PP_1)$ to $\partial(G_2,\PP_2)$ that maps the set of parabolic points of $\partial(G_1,\PP_1)$ bijectively onto the set of parabolic points of $\partial(G_2,\PP_2)$. Moreover, if parabolic point $v$ of $\partial(G_1,\PP_1)$ is labelled by some peripheral left coset $g_1P_1$ in $G_1$ and the parabolic point $f(v)$ of $\partial(G_2,\PP_2)$ is labelled by some peripheral left coset $g_2P_2$ in $G_2$, then $P_1$ and $P_2$ are quasi-isometric.
\end{thm}

\begin{proof}
Fix generating sets $S_1$ and $S_2$ as in Definition \ref{defn001} for $G_1$ and $G_2$ respectively, then there is a quasi-isometry $q\!:\Gamma(G_1,S_1)\to\Gamma(G_2,S_2)$. Since the peripheral structures are minimal, the map $q$ takes a peripheral left coset $g_1P_1$ of $G_1$ to within a uniform bounded distance of the corresponding peripheral left coset $g_2P_2$ of $G_2$ by Theorem 4.1 in \cite{MR2501302}. In particular, $P_1$ and $P_2$ are quasi-isometric. Using the proof of Theorems 6.3 in \cite{MR3152211}, we can extend $q$ to the quasi-isometry $\hat{q}\!:X(G_1,\PP_1,S_1)\to X(G_2,\PP_2,S_2)$ between cusped spaces such that $\hat{q}$ restricts to a quasi-isometric embedding on each individual horoball of $X(G_1,\PP_1,S_1)$ and the image of the horoball lies in some neighborhood of a horoball of $X(G_2,\PP_2,S_2)$. Therefore, there is a homeomorphism $f$ induced by $\hat{q}$ from $\partial(G_1,\PP_1)$ to $\partial(G_2,\PP_2)$ that maps the set of parabolic points of $\partial(G_1,\PP_1)$ bijectively onto the set of parabolic points of $\partial(G_2,\PP_2)$. Moreover, if the parabolic point $v$ of $\partial(G_1,\PP_1)$ is labelled by some peripheral left coset $g_1P_1$ in $G_1$ and the parabolic point $f(v)$ of $\partial(G_2,\PP_2)$ is labelled by some peripheral left coset $g_2P_2$ in $G_2$, then by the above observation $P_1$ and $P_2$ are quasi-isometric.
\end{proof}

\begin{defn}[\cite{MR1837220}]
\label{defn:splitting}
Let $G$ be a group. By a \emph{splitting} of $G$, over a given class of subgroups, we mean a presentation of $G$ as a finite graph of groups, where each edge group belongs to this class. Such a splitting is said to be \emph{relative} to another class $\PP$ of subgroups if each element of $\PP$ is conjugate into one of the vertex groups. A splitting is said to be \emph{trivial} if there exists a vertex group equal to $G$.

Assume $G$ is hyperbolic relative to a collection $\PP$. A \emph{peripheral splitting} of $(G,\PP)$ is a representation of $G$ as a finite bipartite graph of groups, where $\PP$ consists of precisely of the (conjugacy classes of) vertex groups of one color. Obviously, any peripheral splitting of $(G,\PP)$ is relative to $\PP$ and over subgroups of elements of $\PP$. Peripheral splittings of $(G,\PP)$ are closely related to cut points in the Bowditch boundary $\partial(G,\PP)$ (\cite{MR1837220}).
\end{defn}

\begin{defn}
Given a compact connected metric space $X$, a point $x \in X$ is a \emph{global cut point} (or just simply \emph{cut point}) if $X-\{x\}$ is not connected. If $\{a, b\}\subset X$ contains no cut points and $X-\{a, b\}$ is not connected, then $\{a, b\}$ is a \emph{cut pair}. A point $x \in X$ is a \emph{local cut point} if $X-\{x\}$ is not connected, or $X-\{x\}$ is connected and has more than one end. 
\end{defn}

\subsection{$\CAT(0)$ spaces with isolated flats}

In this section, we discuss the work of Hruska-Kleiner \cite{MR2175151} on $\CAT(0)$ spaces with isolated flats.

\begin{defn}
A \emph{$k$--flat} in a $\CAT(0)$ space $X$ is an isometrically embedded copy of Euclidean space $\EE^k$ for some $k \geq 2$. In particular, note that a geodesic line is not considered to be a flat.
\end{defn}

\begin{defn}
Let $X$ be a $\CAT(0)$ space, $G$ a group acting geometrically on $X$, and $\mathcal{F}$ a $G$--invariant set of flats in $X$. We say that $X$ \emph{has isolated flats with respect to $\mathcal{F}$} if the following two conditions hold.
\begin{enumerate}
\item There is a constant $D$ such that every flat $F \subset X$ lies in a $D$--neighborhood of some $F' \in \mathcal{F}$.
\item For each positive $r < \infty$ there is a constant $\rho = \rho(r) < \infty$ so that for any two distinct flats $F, F' \in \mathcal{F}$ we have $\diam\bigl(N_r(F) \cap N_r(F′)\bigr)<\rho$.
\end{enumerate}
We say $X$ has \emph{isolated flats} if it has isolated flats with respect to some $G$--invariant set of flats.
\end{defn}

\begin{thm}[\cite{MR2175151}]
Suppose $X$ has isolated flats with respect to $\mathcal{F}$. For each $F \in \mathcal{F}$ the stabilizer $\Stab_G(F)$ is virtually abelian and acts cocompactly on $F$. The set of stabilizers of flats $F \in \mathcal{F}$ is precisely the set of maximal virtually abelian subgroups of $G$ of rank at least two. These stabilizers lie in only finitely many conjugacy classes.
\end{thm}

\begin{thm}[\cite{MR2175151}]
Let $X$ have isolated flats with respect to $\mathcal{F}$. Then $G$ is relatively hyperbolic with respect to the collection of all maximal
virtually abelian subgroups of rank at least two.
\end{thm}

The previous theorem also has the following converse.

\begin{thm}[\cite{MR2175151}]
\label{hruskaifp}
Let $G$ be a group acting geometrically on a $\CAT(0)$ space $X$. Suppose $G$ is relatively hyperbolic with respect to a family of virtually
abelian subgroups. Then $X$ has isolated flats
\end{thm}

A group $G$ that admits an action on a $\CAT(0)$ space with isolated flats has a ``well-defined'' $\CAT(0)$ boundary, often denoted by $\partial G$, by the following theorem.

\begin{thm}[\cite{MR2175151}]
Let $G$ act properly, cocompactly, and isometrically on $\CAT(0)$ spaces $X$ and $Y$. If $X$ has isolated flats, then so does $Y$, and
there is a $G$--equivariant homeomorphism $\partial X \to \partial Y$.
\end{thm}

\subsection{Right-angled Coxeter groups and their relatively hyperbolic structures}

In this section, we review the concepts of right-angled Coxeter groups and Davis complexes. We also review the work of Caprace \cite{MR2665193, MR3450952} and Behrstock-Hagen-Sisto \cite{MR3623669} on peripheral structures of relatively hyperbolic right-angled Coxeter groups.

\begin{defn}
Given a finite simplicial graph $\Gamma$, the associated \emph{right-angled Coxeter group} $G_\Gamma$ is generated by the set $S$ of vertices of $\Gamma$ and has relations $s^2 = 1$ for all $s$ in $S$ and $st = ts$ whenever $s$ and $t$ are adjacent vertices. Graph $\Gamma$ is the \emph{defining graph} of a right-angled Coxeter group $G_\Gamma$ and its flag complex $\Delta=\Delta(\Gamma)$ is the \emph{defining nerve} of the group. Therefore, sometimes we also denote the right-angled Coxeter group $G_\Gamma$ by $G_\Delta$ where $\Delta$ is the flag complex of $\Gamma$. 

Let $S_1$ be a subset of $S$. The subgroup of $G_\Gamma$ generated by $S_1$ is a right-angled Coxeter group $G_{\Gamma_1}$, where $\Gamma_1$ is the induced subgraph of $\Gamma$ with vertex set $S_1$ (i.e. $\Gamma_1$ is the union of all edges of $\Gamma$ with both endpoints in $S_1$). The subgroup $G_{\Gamma_1}$ is called a \emph{special subgroup} of $G_\Gamma$.
\end{defn}

\begin{rem}
The right-angled Coxeter group $G_{\Gamma}$ is one-ended if and only if $\Gamma$ is not equal to a complete graph, is connected and has no separating complete subgraphs (see Theorem 8.7.2 in \cite{MR2360474}).
\end{rem}

\begin{defn}
Given a finite simplicial graph $\Gamma$, the associated \emph{Davis complex} $\Sigma_\Gamma$ is a cube complex constructed as follows. For every $k$--clique, $T \subset \Gamma$, the special subgroup $G_T$ is isomorphic to the direct product of $k$ copies of $\field{Z}_2$. Hence, the Cayley graph of $G_T$ is isomorphic to the 1--skeleton of a $k$--cube. The Davis complex $\Sigma_\Gamma$ has 1--skeleton the Cayley graph of $G_\Gamma$, where edges are given unit length. Additionally, for each $k$--clique, $T \subset \Gamma$, and coset $gG_T$, we glue a unit $k$--cube to $gG_T \subset\Sigma_\Gamma$. The Davis complex $\Sigma_\Gamma$ is a $\CAT(0)$ space and the group $G_\Gamma$ acts properly and cocompactly on the Davis complex $\Sigma_\Gamma$ (see \cite{MR2360474}).
\end{defn}

\begin{thm}[Theorem A' in \cite{MR2665193, MR3450952}]
\label{th1}
Let $\Gamma$ be a simplicial graph and $\JJ$ be a collection of induced subgraphs of $\Gamma$. Then the right-angled Coxeter group $G_\Gamma$ is hyperbolic relative to the collection $\PP=\set{G_J}{J\in\JJ}$ if and only if the following three conditions hold:
\begin{enumerate}
\item If $\sigma$ is an induced 4-cycle of $\Gamma$, then $\sigma$ is an induced 4-cycle of some $J\in \JJ$.
\item For all $J_1$, $J_2$ in $\JJ$ with $J_1\neq J_2$, the intersection $J_1\cap J_2$ is empty or $J_1\cap J_2$ is a complete subgraph of $\Gamma$.
\item If a vertex $s$ commutes with two non-adjacent vertices of some $J$ in $\JJ$, then $s$ lies in $J$.
\end{enumerate}
\end{thm}


\begin{thm}[Theorem I in \cite{MR3623669}]
\label{n1}
Let $\mathcal{T}$ be the class consisting of the finite simplicial graphs $\Lambda$ such that $G_\Lambda$ is strongly algebraically thick. Then for any finite simplicial graph $\Gamma$ either: $\Gamma \in \mathcal{T}$, or there exists a collection $\JJ$ of induced subgraphs of $\Gamma$ such that $\JJ \subset \mathcal{T}$ and $G_\Gamma$ is hyperbolic relative to the collection $\PP=\set{G_J}{J \in \JJ}$ and this peripheral structure is minimal.
\end{thm}

\begin{rem}
In Theorem \ref{n1} we use the notion of \emph{strong algebraic thickness} which is introduced in \cite{MR3421592} and is a sufficient condition for a group to be non-hyperbolic relative to any collection of proper subgroups. We refer the reader to \cite{MR3421592} for more details. 
\end{rem}

\subsection{Divergence of right-angled Coxeter groups and $3$--manifold groups} Roughly speaking, divergence is a quasi-isometry invariant that measures the circumference of a ball of radius $n$ as a function of $n$. We refer the reader to \cite{MR1254309} for a precise definition. In this section, we state some theorems about divergence of certain right-angled Coxeter groups and $3$-manifold groups which will be used later in this paper.

\subsubsection{Divergence of right-angled Coxeter groups}
\begin{thm}[\cite{MR3623669,Sisto}]
\label{bh}
The divergence of a right-angled Coxeter group is either exponential (if the group is relatively hyperbolic) or bounded above by a polynomial (if the group is strongly algebraically thick).
\end{thm}

Before continuing, we will take a brief detour to define a property of graphs that will be relevant to our study of right-angled Coxeter groups. Given a graph $\Gamma$, define $\Gamma^4$ as the graph whose vertices are induced 4--cycles of $\Gamma$. Two vertices in $\Gamma^4$ are adjacent if and only if the corresponding induced 4-cycles in $\Gamma$ have two nonadjacent vertices in common. 

\begin{defn}[Constructed from squares]\label{defn:CFS}
A graph $\Gamma$ is \emph{$\mathcal{CFS}$} if $\Gamma$ is the join $\Omega*K$ where $K$ is a (possibly empty) clique and $\Omega$ is a non-empty subgraph such that $\Omega^4$ has a connected component $T$ such that every vertex of $\Omega$ is contained in a $4$--cycle that is a vertex of $T$. If $\Gamma$ is $\mathcal{CFS}$, then we will say that the right-angled Coxeter group $G_\Gamma$ is $\mathcal{CFS}$. 
\end{defn}

\begin{thm}[\cite{MR3314816, IL, MR3801467}]
\label{dt}
Let $\Gamma$ be a finite, simplicial, connected graph which has no separating complete subgraph. Let $G_\Gamma$ be the associated right-angled Coxeter group.
\begin{enumerate}
\item The group $G_\Gamma$ has linear divergence if and only if $\Gamma$ is a join of two graphs of diameters at least $2$.
\item The group $G_\Gamma$ has quadratic divergence if and only if $\Gamma$ is $\mathcal{CFS}$ and is not a join of two graphs of diameters at least $2$.
\end{enumerate}
\end{thm}

\subsubsection{Divergence of $3$--manifold groups}

Let $M$ be a compact, connected, orientable $3$--manifold with empty or toroidal boundary. The $3$--manifold $M$ is \emph{geometric} if its interior admits a geometric structure in the sense of Thurston. These structures are the $3$--sphere, Euclidean $3$--space, hyperbolic $3$-space, $S^{2} \times \R$, $\mathbb{H}^{2} \times \R$, $\widetilde{SL(2,\R)}$, Nil and Sol. We note that a geometric $3$--manifold $M$ is Seifert fibered if its geometry is neither Sol nor hyperbolic.
{A non-geometric $3$--manifold can be cut along tori called \emph{JSJ tori} so that each component is a hyperbolic 3-manifold or a Seifert fibered space (\cite{MR539411}, \cite{MR551744}).  Each component of this \emph{JSJ decomposition} is a \emph{piece}. The manifold $M$ is called a \emph{graph manifold} if all the pieces are Seifert fibered, otherwise it is a \emph{mixed manifold}.}

\begin{thm}[Gersten \cite{Gersten94}, Kapovich--Leeb \cite{KL98}]
\label{thm:GKL:divergence}
Let $M$ be a non-geometric manifold. Then $M$ is a graph manifold if and only if the divergence of $\pi_1(M)$ is quadratic, and $M$ is a mixed manifold if and only if the divergence of $\pi_1(M)$ is exponential.
\end{thm}

\begin{rem}
\label{rem:seifertlinear}
Let $M$ be a compact, orientable $3$--manifold whose fundamental group has linear divergence. 
Then $M$ has the geometry of Sol, or $M$ is a Seifert manifold. However, if the universal cover $\tilde{M}$ of $M$ is a fattened tree crossed with $\R$, then $M$ must be a Seifert manifold.
\end{rem}

\subsection{Two-colored graphs and their bisimilarity classes}
\label{bisimilarity}
In this section, we briefly discuss two-colored graphs and the bisimilar equivalence relation on them. These materials are introduced by Behrstock-Neumann \cite{MR2376814} to study the quasi-isometry classification of graph manifolds. In this paper, we also use two-colored graphs and the bisimilar equivalence relation to study the quasi-isometry classification of right-angled Coxeter groups $G_\Gamma$ whose defining graphs $\Gamma$ are $\mathcal{CFS}$ and satisfy Standing Assumptions.

\begin{defn}
\label{weaklycover}
A \emph{two-colored graph} is a graph $\Gamma$ with a ``coloring'' $c:V(\Gamma) \to \{b,w\}$.
A \emph{weak covering} of a two-colored graph $\Gamma'$ is a graph homomorphism $f:\Gamma \to \Gamma'$ which respects colors and has the property that for each $v\in V(\Gamma)$ and for each edge $e' \in E(\Gamma')$ at $f(v)$, there exists an $e\in E(\Gamma)$ at $v$ with $f(e)=e'$.
\end{defn}

\begin{defn}
\label{bisimilarity1}
Two-colored graphs $\Gamma_1$ and $\Gamma_2$ are \emph{bisimilar}, written $\Gamma_1\sim\Gamma_2$ if $\Gamma_1$ and $\Gamma_2$ weakly cover some common two-colored graph.
\end{defn}

\begin{prop}[Proposition 4.3 in \cite{MR2376814}]
The bisimilarity relation $\sim$ is an equivalence relation. Moreover, each equivalence class has a unique minimal element up to isomorphism.
\end{prop}

\section{Bowditch boundaries of relatively hyperbolic right-angled Coxeter groups}
\label{visual splittings and Bowditch boundary}

In this section, we give descriptions of cut points and non-parabolic cut pairs of Bowditch boundaries of relatively hyperbolic right-angled Coxeter groups (see Theorems \ref{intro2} and \ref{intro3}). The Bowditch boundary of relatively hyperbolic groups with minimal peripheral structures is a quasi-isometry invariant (see Theorem~\ref{main}). So, these results can be applied, in certain cases, to differentiate two relatively hyperbolic RACGs in terms of quasi-isometry equivalence.

In \cite{MR3143594}, the third author investigates the connection between the Bowditch boundary of a relatively hyperbolic group $(G,\PP)$ and the boundary of a $\CAT(0)$ space $X$ on which $G$ acts geometrically. For relatively hyperbolic right-angled Coxeter groups, the relevant result from \cite{MR3143594} can be stated as follows:

\begin{thm}[Tran \cite{MR3143594}]
\label{cool}
Let $\Gamma$ be a finite simplicial graph. Assume that the right-angled Coxeter group $G_\Gamma$ is relatively hyperbolic with respect to a collection $\PP$ of its subgroups. Then the Bowditch boundary $\partial (G_\Gamma,\PP)$ is obtained from the $\CAT(0)$ boundary $\partial \Sigma_\Gamma$ by identifying the limit set of each peripheral left coset to a point. Moreover, this quotient map is $G_\Gamma$-equivariant.
\end{thm}

We now introduce some definitions concerning defining graphs of right-angled Coxeter groups that we will use to ``visualize'' cut points and non-parabolic cut points in the Bowditch boundary.

\begin{defn}
Let $\Gamma_1$ and $\Gamma_2$ be two graphs, the \emph{join} of $\Gamma_1$ and $\Gamma_2$, denoted $\Gamma_1*\Gamma_2$, is the graph obtained by connecting every vertex of $\Gamma_1$ to every vertex of $\Gamma_2$ by an edge. If $\Gamma_2$ consists of distinct vertices $u$ and $v$, then the join $\Gamma_1*\{u,v\}$ is the \emph{suspension} of $\Gamma_1$.
\end{defn}

\begin{defn}
Let $\Gamma$ be a simplicial graph. A pair of non-adjacent vertices $\{a, b\}$ in $\Gamma$ is called a \emph{cut pair} if $\{a, b\}$ separates $\Gamma$. An induced subgraph $\Gamma_1$ of $\Gamma$ is a \emph{complete subgraph suspension} if $\Gamma_1$ is a suspension of a complete subgraph $\sigma$ of $\Gamma$. If $\sigma$ is a single vertex, then $\Gamma_1$ is a \emph{vertex suspension}.

An induced subgraph $\Gamma_1$ of $\Gamma$ is \emph{separating} if $\Gamma_1$ separates $\Gamma$. Thus, we may also consider a cut pair as a separating complete subgraph suspension which is a suspension of the empty graph.
\end{defn}

We will need the following lemma to visualize cut points in the Bowditch boundary of a relatively hyperbolic right-angled Coxeter group. 

\begin{lem}
\label{le1}
Let $\Gamma$ be a simplicial graph and $\JJ$ a collection of induced proper subgraphs of $\Gamma$. Assume that the right-angled Coxeter group $G_\Gamma$ is one-ended, hyperbolic relative to the collection $\PP=\set{G_J}{J\in\JJ}$, and suppose each subgroup in $\PP$ is one-ended. Let $J_0$ be an element in $\JJ$ such that some induced subgraph of $J_0$ separates the graph $\Gamma$. Then we can write $\Gamma=\Gamma_1\cup \Gamma_2$ such that the following conditions hold:
\begin{enumerate}
\item $\Gamma_1$, $\Gamma_2$ are both proper induced subgraphs of $\Gamma$;
\item $\Gamma_1\cap\Gamma_2$ is an induced subgraph of $J_0$.
\item For each $J$ in $\JJ$, $J$ lies completely inside either $\Gamma_1$ or $\Gamma_2$
\end{enumerate}
\end{lem}

\begin{proof}
Since each subgroup in $\PP$ is one-ended, each graph $J$ in $\JJ$ is connected and $J$ has no separating complete subgraph by Theorem 8.7.2 in \cite{MR2360474}. Let $L$ be an induced subgraph of $J_0$ that separates the graph $\Gamma$. Let $\Gamma'_1$ be $L$ together with some of the components of $\Gamma-L$, and let $\Gamma'_2$ be $L$ together with the remaining components of $\Gamma-L$. Then, $\Gamma'_1$, $\Gamma'_2$ are both proper induced subgraphs of $\Gamma$, $L=\Gamma'_1\cap\Gamma'_2$, and $\Gamma=\Gamma'_1\cup\Gamma'_2$. Since $J_0$ is a proper subgraph of $\Gamma$, $\Gamma'_1-J_0 \neq \emptyset$ or $\Gamma'_2-J_0\neq \emptyset$ (say $\Gamma'_1-J_0\neq \emptyset$).

Let $\Gamma_1=\Gamma'_1$, $\Gamma_2=\Gamma'_2\cup J_0$. Then, $\Gamma_1$ is an induced proper subgraph of $\Gamma$. We now prove that $\Gamma_2$ is also an induced proper subgraph. Choose a vertex $w\in \Gamma_1-J_0=\Gamma'_1-J_0$. Since $\Gamma_1\cap\Gamma_2=\Gamma'_1\cap(\Gamma'_2\cup J_0)\subset J_0$, the vertex $w$ does not belong to $\Gamma_2$. Therefore, $\Gamma_2$ is a proper subgraph of $\Gamma$. We now prove that $\Gamma_2$ is induced. Let $e$ be an arbitrary edge with endpoints $u$ and $v$ in $\Gamma_2$. If $e$ is an edge of $\Gamma'_2$, then $e$ is the edge of $\Gamma_2$. Otherwise, $e$ is an edge of $\Gamma'_1=\Gamma_1$, because $\Gamma=\Gamma'_1\cup\Gamma'_2$. In particular, $u$ and $v$ are also the vertices of $\Gamma_1$. Again, $\Gamma_1\cap\Gamma_2$ is a subgraph of $J_0$. Then $u$ and $v$ are vertices of $J_0$. Therefore, $e$ is an edge of $J_0$, because $J_0$ is an induced subgraph. Thus, $e$ is also an edge of $\Gamma_2$. Thus, $\Gamma_2$ is an induced subgraph. This implies that $\Gamma_1\cap\Gamma_2$ is an induced subgraph. We already checked that $\Gamma_1\cap\Gamma_2$ is a subgraph of $J_0$, and by construction $\Gamma=\Gamma_1\cup\Gamma_2$.

Now, we prove that for each $J$ in $\JJ$, $J$ lies completely inside either $\Gamma_1$ or $\Gamma_2$. By the construction, $J_0$ is a subgraph of $\Gamma_2$. Therefore, we only need to check the case where $J\neq J_0$. It suffices to show that $J$ lies completely inside either $\Gamma'_1$ or $\Gamma'_2$. By Theorem \ref{th1}, for each $J\neq J_0$ in $\JJ$ the intersection $J\cap J_0$ is empty or it is a complete subgraph of $\Gamma$. Also the intersection $J\cap L$ is an induced subgraph of $J\cap J_0$ if $J\cap L\neq \emptyset$. Therefore, $J\cap L$ is empty or it is a complete subgraph of $\Gamma$. Recall that each graph in $\JJ$ is connected and has no separating complete subgraph by our assumption that the peripheral subgroups are $1$--ended. Therefore, $J-L=J-(J\cap L)$ is connected for each $J\neq J_0$ in $\JJ$. By the construction $J-L$ lies completely inside either $\Gamma'_1$ or $\Gamma'_2$. Thus, $J$ also lies completely inside either $\Gamma'_1$ or $\Gamma'_2$. Therefore, $J$ lies completely inside either $\Gamma_1$ or $\Gamma_2$.
\end{proof}



The following theorem describes cut points in the Bowditch boundary of a relatively hyperbolic right-angled Coxeter group using the defining graph.

\begin{thm}
\label{intro2}
Let $\Gamma$ be a simplicial graph and $\JJ$ be a collection of induced proper subgraphs of $\Gamma$. Assume that the right-angled Coxeter group $G_\Gamma$ is one-ended, $G_\Gamma$ is hyperbolic relative to the collection $\PP=\set{G_J}{J\in\JJ}$, and suppose each subgroup in $\PP$ is also one-ended. Then each parabolic point $v_{gG_{J_0}}$ is a global cut point if and only if some induced subgraph of $J_0$ separates the graph $\Gamma$.
\end{thm}

\begin{proof}
We first assume that some induced subgraph of $J_0$ separates the graph $\Gamma$. Let $\Gamma_1$ and $\Gamma_2$ be graphs satisfying all conditions of Lemma~\ref{le1}. 
This implies that $G_\Gamma=G_{\Gamma_1}*_{G_K}G_{\Gamma_2}$, $G_{\Gamma_1}\neq G_\Gamma$, and $G_{\Gamma_2}\neq G_\Gamma$. Since $J$ lies completely inside either $\Gamma_1$ or $\Gamma_2$ for each $J\in\JJ$, each peripheral subgroup in $\PP$ must be a subgroup of $G_{\Gamma_1}$ or $G_{\Gamma_2}$. Therefore, $G_\Gamma$ splits non-trivially relative to $\PP$ over the parabolic subgroup $G_K\leq G_{J_0}$. 

By the claim following Theorem 1.2 of \cite{MR1837220} the parabolic point $v_{G_{J_0}}$ labelled by $G_{J_0}$ is a global cut point of $\partial(G,\PP)$. Also, the group $G_\Gamma$ acts as a group of homeomorphisms on $\partial(G,\PP)$ and $gv_{G_{J_0}}=v_{gG_{J_0}}$. Thus, each parabolic point $v_{gG_{J_0}}$ is also a global cut point.

Now, we assume that some parabolic point $v_{gG_{J_0}}$ is a global cut point. Again, the group $G_\Gamma$ acts as a group of homeomorphisms on $\partial(G,\PP)$ and $gv_{G_{J_0}}=v_{gG_{J_0}}$. Therefore, $v_{G_{J_0}}$ is also a global cut point. Thus the maximal peripheral splitting $\mathcal{G}$ of $(G,\PP)$ is non-trivial, and $G_{J_0}$ is a vertex stabilizer which is adjacent to a component vertex group in $\mathcal{G}$ (see \cite{MR1837220} for details concerning maximal peripheral splittings). So, by Theorem 3.3 of \cite{HM1} $G_\Gamma$ splits non-trivially over a subgroup $H$ of $G_{J_0}$. Theorem 1 of \cite{MR2466022} implies that there is some induced subgraph $K$ of $\Gamma$ which separates $\Gamma$ such that $G_K$ is contained in some conjugate of $H$. Therefore, $G_K$ is also contained in some conjugate of the peripheral subgroup $G_{J_0}$. Moreover, $G_K$ and $G_{J_0}$ are both special subgroups of $G_\Gamma$. Thus, $K$ is an induced subgraph of $J_0$. (This is a standard fact, and we leave the details to the reader.)
\end{proof}



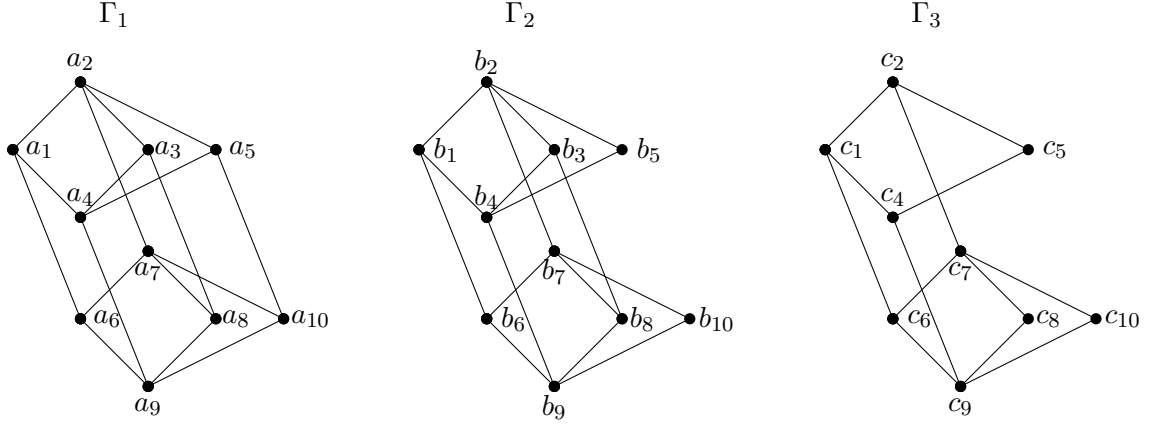
\begin{figure}
\begin{tikzpicture}[scale=0.9]

\draw (-11,1) node[circle,fill,inner sep=1.5pt, color=black](1){} -- (-10,2) node[circle,fill,inner sep=1.5pt, color=black](1){}-- (-9,1) node[circle,fill,inner sep=1.5pt, color=black](1){}-- (-10,0) node[circle,fill,inner sep=1.5pt, color=black](1){} -- (-11,1) node[circle,fill,inner sep=1.5pt, color=black](1){};

\draw (-10,2) node[circle,fill,inner sep=1.5pt, color=black](1){} -- (-8,1) node[circle,fill,inner sep=1.5pt, color=black](1){}-- (-10,0) node[circle,fill,inner sep=1.5pt, color=black](1){};

\draw (-10,-1.5) node[circle,fill,inner sep=1.5pt, color=black](1){} -- (-9,-0.5) node[circle,fill,inner sep=1.5pt, color=black](1){}-- (-8,-1.5) node[circle,fill,inner sep=1.5pt, color=black](1){}-- (-9,-2.5) node[circle,fill,inner sep=1.5pt, color=black](1){} -- (-10,-1.5) node[circle,fill,inner sep=1.5pt, color=black](1){};

\draw (-9,-0.5) node[circle,fill,inner sep=1.5pt, color=black](1){} -- (-7,-1.5) node[circle,fill,inner sep=1.5pt, color=black](1){}-- (-9,-2.5) node[circle,fill,inner sep=1.5pt, color=black](1){};

\draw (-11,1) node[circle,fill,inner sep=1.5pt, color=black](1){} -- (-10,-1.5) node[circle,fill,inner sep=1.5pt, color=black](1){};
\draw (-10,2) node[circle,fill,inner sep=1.5pt, color=black](1){} -- (-9,-0.5) node[circle,fill,inner sep=1.5pt, color=black](1){};
\draw (-9,1) node[circle,fill,inner sep=1.5pt, color=black](1){} -- (-8,-1.5) node[circle,fill,inner sep=1.5pt, color=black](1){};
\draw (-10,0) node[circle,fill,inner sep=1.5pt, color=black](1){} -- (-9,-2.5) node[circle,fill,inner sep=1.5pt, color=black](1){};
\draw (-8,1) node[circle,fill,inner sep=1.5pt, color=black](1){} -- (-7,-1.5) node[circle,fill,inner sep=1.5pt, color=black](1){};

\node at (-10.6,1) {$a_1$}; \node at (-10,2.3) {$a_2$}; \node at (-8.7,1) {$a_3$}; \node at (-10,0.3) {$a_4$}; \node at (-7.6,1) {$a_5$};
\node at (-9.6,-1.5) {$a_6$}; \node at (-9,-0.8) {$a_7$}; \node at (-7.7,-1.5) {$a_8$}; \node at (-9,-2.8) {$a_9$}; \node at (-6.6,-1.5) {$a_{10}$};

\node at (-9.5,3) {$\Gamma_1$};

\draw (-5,1) node[circle,fill,inner sep=1.5pt, color=black](1){} -- (-4,2) node[circle,fill,inner sep=1.5pt, color=black](1){}-- (-3,1) node[circle,fill,inner sep=1.5pt, color=black](1){}-- (-4,0) node[circle,fill,inner sep=1.5pt, color=black](1){} -- (-5,1) node[circle,fill,inner sep=1.5pt, color=black](1){};

\draw (-4,2) node[circle,fill,inner sep=1.5pt, color=black](1){} -- (-2,1) node[circle,fill,inner sep=1.5pt, color=black](1){}-- (-4,0) node[circle,fill,inner sep=1.5pt, color=black](1){};

\draw (-4,-1.5) node[circle,fill,inner sep=1.5pt, color=black](1){} -- (-3,-0.5) node[circle,fill,inner sep=1.5pt, color=black](1){}-- (-2,-1.5) node[circle,fill,inner sep=1.5pt, color=black](1){}-- (-3,-2.5) node[circle,fill,inner sep=1.5pt, color=black](1){} -- (-4,-1.5) node[circle,fill,inner sep=1.5pt, color=black](1){};

\draw (-3,-0.5) node[circle,fill,inner sep=1.5pt, color=black](1){} -- (-1,-1.5) node[circle,fill,inner sep=1.5pt, color=black](1){}-- (-3,-2.5) node[circle,fill,inner sep=1.5pt, color=black](1){};

\draw (-5,1) node[circle,fill,inner sep=1.5pt, color=black](1){} -- (-4,-1.5) node[circle,fill,inner sep=1.5pt, color=black](1){};
\draw (-4,2) node[circle,fill,inner sep=1.5pt, color=black](1){} -- (-3,-0.5) node[circle,fill,inner sep=1.5pt, color=black](1){};
\draw (-3,1) node[circle,fill,inner sep=1.5pt, color=black](1){} -- (-2,-1.5) node[circle,fill,inner sep=1.5pt, color=black](1){};
\draw (-4,0) node[circle,fill,inner sep=1.5pt, color=black](1){} -- (-3,-2.5) node[circle,fill,inner sep=1.5pt, color=black](1){};

\node at (-4.6,1) {$b_1$}; \node at (-4,2.3) {$b_2$}; \node at (-2.7,1) {$b_3$}; \node at (-4,0.3) {$b_4$}; \node at (-1.6,1) {$b_5$};
\node at (-3.6,-1.5) {$b_6$}; \node at (-3,-0.8) {$b_7$}; \node at (-1.7,-1.5) {$b_8$}; \node at (-3,-2.8) {$b_9$}; \node at (-0.6,-1.5) {$b_{10}$};

\node at (-3.5,3) {$\Gamma_2$};

\draw (1,1) node[circle,fill,inner sep=1.5pt, color=black](1){} -- (2,2) node[circle,fill,inner sep=1.5pt, color=black](1){};
\draw (1,1) node[circle,fill,inner sep=1.5pt, color=black](1){} -- (2,0) node[circle,fill,inner sep=1.5pt, color=black](1){};

\draw (2,2) node[circle,fill,inner sep=1.5pt, color=black](1){} -- (4,1) node[circle,fill,inner sep=1.5pt, color=black](1){}-- (2,0) node[circle,fill,inner sep=1.5pt, color=black](1){};

\draw (2,-1.5) node[circle,fill,inner sep=1.5pt, color=black](1){} -- (3,-0.5) node[circle,fill,inner sep=1.5pt, color=black](1){}-- (4,-1.5) node[circle,fill,inner sep=1.5pt, color=black](1){}-- (3,-2.5) node[circle,fill,inner sep=1.5pt, color=black](1){} -- (2,-1.5) node[circle,fill,inner sep=1.5pt, color=black](1){};

\draw (3,-0.5) node[circle,fill,inner sep=1.5pt, color=black](1){} -- (5,-1.5) node[circle,fill,inner sep=1.5pt, color=black](1){}-- (3,-2.5) node[circle,fill,inner sep=1.5pt, color=black](1){};

\draw (1,1) node[circle,fill,inner sep=1.5pt, color=black](1){} -- (2,-1.5) node[circle,fill,inner sep=1.5pt, color=black](1){};
\draw (2,2) node[circle,fill,inner sep=1.5pt, color=black](1){} -- (3,-0.5) node[circle,fill,inner sep=1.5pt, color=black](1){};
\draw (2,0) node[circle,fill,inner sep=1.5pt, color=black](1){} -- (3,-2.5) node[circle,fill,inner sep=1.5pt, color=black](1){};

\node at (1.4,1) {$c_1$}; \node at (2,2.3) {$c_2$}; 
\node at (2,0.3) {$c_4$}; \node at (4.4,1) {$c_5$};
\node at (2.4,-1.5) {$c_6$}; \node at (3,-0.8) {$c_7$}; \node at (4.3,-1.5) {$c_8$}; \node at (3,-2.8) {$c_9$}; \node at (5.4,-1.5) {$c_{10}$};

\node at (2.5,3) {$\Gamma_3$};

\end{tikzpicture}

\caption{The three groups $G_{\Gamma_1}$, $G_{\Gamma_2}$, and $G_{\Gamma_3}$ are pairwise not quasi-isometric because the Bowditch boundaries with respect to their minimal peripheral structures are pairwise not homeomorphic.}
\label{afirst}
\end{figure}

Now, we discuss a few examples related to cut points in Bowditch boundaries of relatively hyperbolic right-angled Coxeter groups. These examples illustrate an application of Theorem \ref{intro2} to the problem of quasi-isometry classification of right-angled Coxeter groups.

\begin{exmp}
\label{ex1}
Let $\Gamma_1$, $\Gamma_2$, and $\Gamma_3$ be the graphs in Figure \ref{afirst}. Observe that all groups $G_{\Gamma_i}$ are one-ended. We will prove that groups $G_{\Gamma_1}$, $G_{\Gamma_2}$, and $G_{\Gamma_3}$ are not pairwise quasi-isometric by investigating their minimal peripheral structures.

In $\Gamma_1$, let $K_1^{(1)}$ and $K_1^{(2)}$ be induced subgraphs generated by $\{a_1,a_2,a_3,a_4,a_5\}$ and $\{a_6,a_7,a_8,a_9,a_{10}\}$, respectively. Observe that $\Gamma_1$ has only six induced 4-cycles which are not subgraphs of $K_1^{(1)}$ and $K_1^{(2)}$. Denote these cycles by $L_1^{(i)}$ ($i=1,2,\cdots, 6$). Let $\JJ_1$ be the set of all graphs $L_1^{(i)}$ and $K_1^{(j)}$. By Theorems \ref{th1} and \ref{n1}, $\JJ_1$ is the minimal peripheral structure of $\Gamma_1$. Moreover, no induced subgraph of a graph in $\JJ_1$ separates $\Gamma_1$. Therefore by Theorem \ref{intro2}, $G_{\Gamma_1}$ is hyperbolic relative to the collection $\PP_1=\set{G_J}{J\in\JJ_1}$ and the Bowditch boundary $\partial (G_{\Gamma_1},\PP_1)$ has no global cut point.

Similarly, let $K_2^{(1)}$ and $K_2^{(2)}$ be the induced subgraphs of $\Gamma_2$ generated by $\{b_1,b_2,b_3,b_4,b_5\}$ and $\{b_6,b_7,b_8,b_9,b_{10}\}$, respectively. Observe that $\Gamma_2$ has only four induced 4-cycles, denoted $L_2^{(i)}$ ($i=1,2,\cdots, 4$), such that each of them is not a subgraph of $K_2^{(1)}$ and $K_2^{(2)}$. Let $\JJ_2$ be the set of all graphs $L_2^{(i)}$ and $K_2^{(j)}$. Then by Theorems \ref{th1} and \ref{n1}, $\JJ_2$ is the minimal peripheral structure of $\Gamma_2$. Moreover, $K_2^{(1)}$ and $K_2^{(2)}$ are the only graphs in $\JJ_2$ which contain induced subgraphs that separate $\Gamma_2$. Therefore, $G_{\Gamma_2}$ is hyperbolic relative to the collection $\PP_2=\set{G_J}{J\in\JJ_2}$, the Bowditch boundary $\partial (G_{\Gamma_2},\PP_2)$ has global cut points and each of them is labelled by some left coset of $G_{K_2^{(1)}}$ or $G_{K_2^{(2)}}$ by Theorem \ref{intro2}.

Finally, let $K_3^{(1)}$ be an induced subgraph of $\Gamma_3$ generated by $\{c_6,c_7,c_8,c_9,c_{10}\}$. The graph $\Gamma_3$ has only three induced 4-cycles, denoted $L_3^{(i)}$ ($i=1,2,3$), such that each of them is not a subgraph of $K_3^{(1)}$. Assume that $L_3^{(1)}$ is the induced 4-cycle generated by $\{c_1,c_2,c_5,c_4\}$. Let $\JJ_3$ be the set of all graphs $L_3^{(i)}$ and $K_3^{(1)}$. Again, by Theorem \ref{th1} and \ref{n1} we have that $\JJ_3$ is the minimal peripheral structure of $\Gamma_3$. Moreover, $K_3^{(1)}$ and $L_3^{(1)}$ are the only graphs in $\JJ_3$ which contain induced subgraphs that separate $\Gamma_3$. Therefore, $G_{\Gamma_3}$ is hyperbolic relative to the collection $\PP_3=\set{G_J}{J\in\JJ_3}$, the Bowditch boundary $\partial (G_{\Gamma_3},\PP_3)$ has global cut points and each of them is labelled by some left coset of $G_{K_3^{(1)}}$ or $G_{L_3^{(1)}}$ by Theorem \ref{intro2}.

Note that all the groups in $\PP_i$ are one-ended. The Bowditch boundary $\partial (G_{\Gamma_1},\PP_1)$ has no global cut point, but the Bowditch boundaries $\partial (G_{\Gamma_2},\PP_2)$ and $\partial (G_{\Gamma_3},\PP_3)$ do. So, $G_{\Gamma_1}$ cannot be quasi-isometric to $G_{\Gamma_2}$ or $G_{\Gamma_3}$. Additionally, the Bowditch boundary $\partial (G_{\Gamma_3},\PP_3)$ has global cut points labelled by a left coset of $G_{L_3^{(1)}}$. Meanwhile, no global cut point of the Bowditch boundary $\partial (G_{\Gamma_2},\PP_2)$ is labelled by the left coset of a peripheral subgroup which is quasi-isometric to $G_{L_3^{(1)}}$. Therefore, $G_{\Gamma_2}$ and $G_{\Gamma_3}$ are not quasi-isometric.
\end{exmp}

In the remainder of this section, we describe non-parabolic cut pairs in Bowditch boundaries of relatively hyperbolic right-angled Coxeter groups in terms of their defining graphs.

\begin{prop}
\label{pi1}
Let $\Gamma$ be a simplicial graph and $\JJ$ be a collection of induced proper subgraphs of $\Gamma$. Assume that the right-angled Coxeter group $G_\Gamma$ is one-ended, hyperbolic relative to the collection $\PP=\set{G_J}{J\in\JJ}$, and suppose each subgroup in $\PP$ is also one-ended. If $\Gamma$ has a separating complete subgraph suspension whose non-adjacent vertices do not lie in the same subgraph $J\in \JJ$, then the $\CAT(0)$ boundary $\partial \Sigma_\Gamma$ has a cut pair and the Bowditch boundary $\partial(G_\Gamma,\PP)$ has a non-parabolic cut pair.
\end{prop}

\begin{proof}
Let $K$ be a separating complete subgraph suspension of $\Gamma$ whose non-adjacent vertices $u$ and $v$ do not both lie in the same subgraph $J\in \JJ$. Let $T$ be the set of all vertices of $\Gamma$ which are both adjacent to $u$ and $v$. Then $T$ is a vertex set of a complete subgraph $\sigma$ of $\Gamma$. Otherwise, the two vertices $u$ and $v$ both lie in the same 4-cycle. Thus, $u$ and $v$ lie in the same subgraph $J\in\JJ$, a contradiction.

Let $\bar{K}=\sigma*\{u,v\}$. We can easily verify the following properties of $\bar{K}$.
\begin{enumerate}
\item For each $J\in\JJ$ the intersection $\bar{K}\cap J$ is empty or a complete subgraph.
\item No vertex outside $\bar{K}$ is adjacent to the unique pair of nonadjacent vertices $\{u,v\}$ of $\bar{K}$.
\item $K$ is an induced subgraph of $\bar{K}$.
\end{enumerate}
Therefore, the collection $\bar{\JJ}=\JJ\cup \{\bar{K}\}$ satisfies all the conditions of Theorem \ref{th1}, which implies that $G_\Gamma$ is hyperbolic relative to the collection $\bar{\PP}=\set{G_J}{J\in\bar{\JJ}}$.

Using an argument similar to that of Lemma \ref{le1}, we can write $\Gamma=\Gamma_1\cup \Gamma_2$ such that the following conditions hold:
\begin{enumerate}
\item $\Gamma_1$, $\Gamma_2$ are both proper induced subgraphs of $\Gamma$;
\item $\Gamma_1\cap\Gamma_2$ is an induced subgraph $L$ of $\bar{K}$.
\item For each $J$ in $\bar{\JJ}$, $J$ lies completely inside either $\Gamma_1$ or $\Gamma_2$
\end{enumerate}
Therefore, we can prove that the Bowditch boundary $\partial (G,\bar{\PP})$ has a global cut point $v_{G_{\bar{K}}}$ stabilized by the subgroup $G_{\bar{K}}$ by using an argument similar to the one in Theorem \ref{intro2}.

By Theorem \ref{cool}, the Bowditch boundary $\partial (G_\Gamma,\bar{\PP})$ is obtained from the $\CAT(0)$ boundary $\partial \Sigma_\Gamma$ by identifying the limit set of each peripheral left coset of a subgroup in $\bar{\PP}$ to a point. Let $f$ be this quotient map. Since $G_{\bar{K}}$ is two-ended, its limit set consists of two points $w_1$ and $w_2$ in $\partial \Sigma_\Gamma$. Therefore, $f(w_1)=f(w_2)=v_{G_{\bar{K}}}$ and $f\bigl(\partial\Sigma_\Gamma-\{w_1,w_2\}\bigr)=\partial(G_\Gamma,\bar{\PP})-\{v_{G_{\bar{K}}}\}$. Since $\partial(G_\Gamma,\bar{\PP})-\{v_{G_{\bar{K}}}\}$ is not connected, the space $\partial\Sigma_\Gamma-\{w_1,w_2\}$ is also not connected. This implies that $\{w_1,w_2\}$ is a cut pair of the $\CAT(0)$ boundary $\partial\Sigma_\Gamma$.

Again, by Theorem \ref{cool} the Bowditch boundary $\partial (G_\Gamma,\PP)$ is obtained from the $\CAT(0)$ boundary $\partial \Sigma_\Gamma$ by identifying the limit set of each peripheral left coset of a subgroup in $\PP$ to a point. Let $h$ be this quotient map. The two points $w_1$ and $w_2$ do not lie in limit sets of peripheral left cosets of subgroups in $\PP$. Therefore, $h(w_1)\neq h(w_2)$ and they are non-parabolic points in the Bowditch boundary $\partial (G_\Gamma,\PP)$. Moreover, $h\bigl(\partial\Sigma_\Gamma-\{w_1,w_2\}\bigr)=\partial(G_\Gamma,\PP)-\{h(w_1),h(w_2)\}$ and the limit set of each peripheral left coset of a subgroup in $\PP$ lies completely inside $\partial\Sigma_\Gamma-\{w_1,w_2\}$.

We observe that for any two points $s_1, s_2 \in \partial \Sigma_\Gamma -\{w_1,w_2\}$ satisfying $h(s_1)=h(s_2)$ the two points $s_1$ and $s_2$ both lie in some limit set $C$ of a peripheral left coset of a subgroup in $\PP$. Also, each subgroup in $\PP$ is one-ended, so $C$ is connected. Therefore, $s_1$ and $s_2$ lie in the same connected component of $\partial \Sigma_\Gamma -\{w_1,w_2\}$. This implies that if $U$ and $V$ are different components of $\partial \Sigma_\Gamma -\{w_1,w_2\}$, then $h(U)\cap h(V)=\emptyset$. Therefore, $\partial(G_\Gamma,\PP)-\{h(w_1),h(w_2)\}$ is not connected. This implies that $\{h(w_1),h(w_2)\}$ is a non-parabolic cut pair of the Bowditch boundary $\partial(G_\Gamma,\PP)$.
\end{proof}

The following theorem describes non-parabolic cut pairs in Bowditch boundaries of relatively hyperbolic right-angled Coxeter groups in terms of their defining graphs.

\begin{thm}
\label{intro3}
Let $\Gamma$ be a simplicial graph and $\JJ$ be a collection of induced proper subgraphs of $\Gamma$. Assume that the right-angled Coxeter group $G_\Gamma$ is one-ended, $G_\Gamma$ is hyperbolic relative to the collection $\PP=\set{G_J}{J\in\JJ}$, and suppose each subgroup in $\PP$ is one-ended. If the Bowditch boundary $\partial (G_\Gamma,\PP)$ has a non-parabolic cut pair, then $\Gamma$ has a separating complete subgraph suspension. Moreover, if $\Gamma$ has a separating complete subgraph suspension whose non-adjacent vertices do not lie in the same subgraph $J \in \JJ$, then the Bowditch boundary $\partial (G_\Gamma,\PP)$ has a non-parabolic cut pair.
\end{thm}

\begin{proof}
Since $G_\Gamma$ is one-ended, $G_\Gamma$ does not split over a finite group by the characterization of one-endedness in Stallings's theorem. Therefore by Proposition 10.1 in \cite{MR2922380} the Bowditch boundary $\partial (G_\Gamma,\PP)$ is connected. If the Bowditch boundary $\partial (G_\Gamma,\PP)$ is a circle, then $G_\Gamma$ is virtually a surface group, and the peripheral subgroups are the boundary subgroups of that surface by Theorem 6B in \cite {MR961162}. This is a contradiction because each peripheral subgroup is one-ended. Therefore, the Bowditch boundary $\partial (G_\Gamma,\PP)$ is not a circle. We now assume that the Bowditch boundary $\partial (G_\Gamma,\PP)$ has a non-parabolic cut pair $\{u,v\}$. Then $u$ is obviously a non-parabolic local cut point. Therefore by Theorem 1.1 in \cite{HM1}, $G_\Gamma$ splits over a two-ended subgroup $H$.


Since $G_\Gamma$ splits over a two-ended subgroup $H$, there is an induced subgraph $K$ of $\Gamma$ which separates $\Gamma$ such that $G_K$ is contained in some conjugate of $H$ by Theorem 1 in \cite{MR2466022}. Because the group $G_\Gamma$ is one-ended, the group $G_K$ is two-ended. This implies that $K$ is a complete subgraph suspension. The remaining conclusion is obtained from Proposition \ref{pi1}.
\end{proof}

Now, we discuss a few examples related to cut points in Bowditch boundaries of relatively hyperbolic right-angled Coxeter groups. These examples illustrate an application of Theorem \ref{intro3} to the problem of quasi-isometry classification of right-angled Coxeter groups.

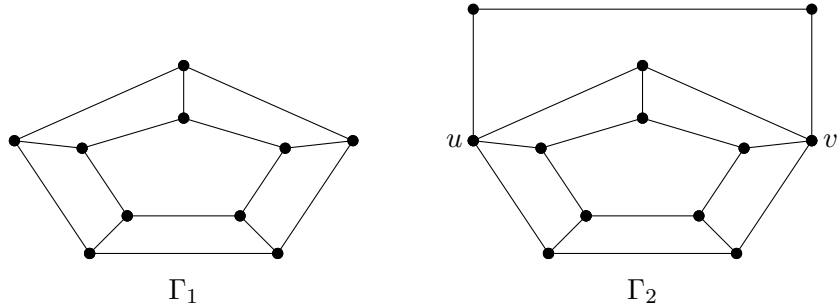
\begin{figure}
\begin{tikzpicture}[scale=0.5]

\draw (4,0) node[circle,fill,inner sep=1.5pt, color=black](1){} -- (2,3) node[circle,fill,inner sep=1.5pt, color=black](1){} -- (6.5,5) node[circle,fill,inner sep=1.5pt, color=black](1){}-- (11,3) node[circle,fill,inner sep=1.5pt, color=black](1){} -- (9,0) node[circle,fill,inner sep=1.5pt, color=black](1){} -- (4,0) node[circle,fill,inner sep=1.5pt, color=black](1){};

\draw (5,1) node[circle,fill,inner sep=1.5pt, color=black](1){} -- (3.8,2.8) node[circle,fill,inner sep=1.5pt, color=black](1){} -- (6.5,3.6) node[circle,fill,inner sep=1.5pt, color=black](1){}-- (9.2,2.8) node[circle,fill,inner sep=1.5pt, color=black](1){} -- (8,1) node[circle,fill,inner sep=1.5pt, color=black](1){} -- (5,1) node[circle,fill,inner sep=1.5pt, color=black](1){};

\draw (4,0) node[circle,fill,inner sep=1.5pt, color=black](1){} -- (5,1) node[circle,fill,inner sep=1.5pt, color=black](1){};
\draw (2,3) node[circle,fill,inner sep=1.5pt, color=black](1){} -- (3.8,2.8) node[circle,fill,inner sep=1.5pt, color=black](1){};
\draw (6.5,5) node[circle,fill,inner sep=1.5pt, color=black](1){} -- (6.5,3.6) node[circle,fill,inner sep=1.5pt, color=black](1){};
\draw (11,3) node[circle,fill,inner sep=1.5pt, color=black](1){} --  (9.2,2.8) node[circle,fill,inner sep=1.5pt, color=black](1){};
\draw (9,0) node[circle,fill,inner sep=1.5pt, color=black](1){} --  (8,1) node[circle,fill,inner sep=1.5pt, color=black](1){};

\node at (6.5,-1) {$\Gamma_1$};

\draw (16.2,0) node[circle,fill,inner sep=1.5pt, color=black](1){} -- (14.2,3) node[circle,fill,inner sep=1.5pt, color=black](1){} -- (18.7,5) node[circle,fill,inner sep=1.5pt, color=black](1){}-- (23.2,3) node[circle,fill,inner sep=1.5pt, color=black](1){} -- (21.2,0) node[circle,fill,inner sep=1.5pt, color=black](1){} -- (16.2,0) node[circle,fill,inner sep=1.5pt, color=black](1){};

\draw (17.2,1) node[circle,fill,inner sep=1.5pt, color=black](1){} -- (16,2.8) node[circle,fill,inner sep=1.5pt, color=black](1){} -- (18.7,3.6) node[circle,fill,inner sep=1.5pt, color=black](1){}-- (21.4,2.8) node[circle,fill,inner sep=1.5pt, color=black](1){} -- (20.2,1) node[circle,fill,inner sep=1.5pt, color=black](1){} -- (17.2,1) node[circle,fill,inner sep=1.5pt, color=black](1){};

\draw (16.2,0) node[circle,fill,inner sep=1.5pt, color=black](1){} --  (17.2,1) node[circle,fill,inner sep=1.5pt, color=black](1){};
\draw (14.2,3) node[circle,fill,inner sep=1.5pt, color=black](1){} --  (16,2.8) node[circle,fill,inner sep=1.5pt, color=black](1){};
\draw (18.7,5) node[circle,fill,inner sep=1.5pt, color=black](1){} --  (18.7,3.6) node[circle,fill,inner sep=1.5pt, color=black](1){};
\draw (23.2,3) node[circle,fill,inner sep=1.5pt, color=black](1){} --  (21.4,2.8) node[circle,fill,inner sep=1.5pt, color=black](1){};
\draw (21.2,0) node[circle,fill,inner sep=1.5pt, color=black](1){} --  (20.2,1) node[circle,fill,inner sep=1.5pt, color=black](1){};

\draw (14.2,3) node[circle,fill,inner sep=1.5pt, color=black](1){} -- (14.2,6.5) node[circle,fill,inner sep=1.5pt, color=black](1){} -- (23.2,6.5) node[circle,fill,inner sep=1.5pt, color=black](1){}-- (23.2,3) node[circle,fill,inner sep=1.5pt, color=black](1){};

\node at (18.7,-1) {$\Gamma_2$};
\node at (13.7,3) {$u$};
\node at (23.7,3) {$v$};

\end{tikzpicture}

\caption{Graph $\Gamma_1$ has no separating complete subgraph suspension while graph $\Gamma_2$ has cut pair $(u,v)$ such that $u$ and $v$ do not lie in the same 4-cycle.}
\label{asecond}
\end{figure}

\begin{exmp}
\label{ex2}
Let $\Gamma_1$ and $\Gamma_2$ be the graphs in Figure \ref{asecond}. Then $G_{\Gamma_1}$ and $G_{\Gamma_2}$ are both one-ended. Let $\JJ_1$ and $\JJ_2$ be the sets of all induced 4-cycles of $\Gamma_1$ and $\Gamma_2$, respectively. By Theorems \ref{th1} and \ref{n1}, the collection $\JJ_i$ is the minimal peripheral structure of $\Gamma_i$ for each $i$. Also, subgroups in each $\PP_i=\set{G_J}{J\in\JJ_i}$ are virtually $\ZZ^2$ and thus one-ended. Moreover, for each $i$ no induced subgraph of a graph in $\JJ_i$ separates $\Gamma_i$. Thus by Theorem \ref{intro2}, both Bowditch boundaries $\partial (G_{\Gamma_1},\PP_1)$ and $\partial (G_{\Gamma_2},\PP_2)$ have no cut points. So in this case, we cannot use cut points to differentiate $G_{\Gamma_1}$ and $G_{\Gamma_2}$ up to quasi-isometry. However, the graph $\Gamma_1$ has no separating complete subgraph suspension, so by Theorem \ref{intro3} the Bowditch boundary $\partial (G_{\Gamma_1},\PP_1)$ has no non-parabolic cut pair. Meanwhile, $\Gamma_2$ has a cut pair $(u,v)$ such that $u$ and $v$ do not lie in the same subgraph in $\JJ_2$. Again, by Theorem \ref{intro3} the Bowditch boundary $\partial (G_{\Gamma_2},\PP_2)$ has a non-parabolic cut pair. By Theorem \ref{main}, $G_{\Gamma_1}$ and $G_{\Gamma_2}$ are not quasi-isometric.
\end{exmp}

\section{Geometric structure of right-angled Coxeter groups with planar defining nerves}
\label{nice3}

In this section, we study the coarse geometry of right-angled Coxeter groups with planar defining nerves. We first analyze the tree structure of planar flag complexes. Then we use this structure to study the relatively hyperbolic structure, group divergence, and manifold structure of right-angled Coxeter groups with planar nerves. Finally, we give a complete quasi-isometry classification of right-angled Coxeter groups which are virtually graph manifold groups.

\begin{defn}
A simplicial complex $\Delta$ is called \emph{flag} if any complete subgraph of the $1$-skeleton of $\Delta$ is the 1-skeleton of a simplex of $\Delta$. Let $\Gamma$ be a finite simplicial graph. The \emph{flag complex} of $\Gamma$ is the flag complex with 1-skeleton $\Gamma$. A simplicial subcomplex $B$ of a simplicial complex $\Delta$ is called \emph{full} if every simplex in $\Delta$ whose vertices all belong to $B$ is itself in $B$.

The flag complex of $\Delta$ is \emph{planar} if it can be embedded into the $2$-dimensional sphere $\field{S}^2$. From now on every time we consider a flag complex it will be as a subspace of the $2$-dimensional sphere $\field{S}^2$. 
\end{defn}

The following lemma provides necessary and sufficient conditions for right-angled Coxeter groups of planar nerves to be one-ended.

\begin{lem}
\label{lem:one ended}
Let $\Delta\subset \field{S}^2$ be a non-simplex planar flag complex. 
Then $G_{\Delta}$ is one-ended if and only if $\Delta$ is connected has no separating vertex and no separating edge.
\end{lem}

\begin{proof}
Assume that our flag complex $\Delta$ is not a simplex. It is known that the right-angled Coxeter group $G_\Delta$ is one-ended if and only if $\Delta$ is connected with no separating simplex. When $\Delta\subset \field{S}^2$ is a connected flag complex and $\Delta$ has a separating simplex, we observe that $\Delta$ has a separating vertex or a separating edge. 
\end{proof}

The following lemma provides necessary and sufficient conditions for when a one-ended right-angled Coxeter group with planar nerve splits over a two-ended subgroup.

\begin{lem}
\label{lem:splitoverZ}
Let $\Delta\subset \field{S}^2$ be a non-simplex connected flag complex with no separating vertex and no separating edge. Then $G_{\Delta}$ splits over a two-ended subgroup if and only if $\Delta$ has a cut pair or a separating induced path of length $2$.
\end{lem}
\begin{proof}
By Lemma~\ref{lem:one ended}, the right-angled Coxeter group $G_{\Delta}$ is one-ended. Hence, $G_\Delta$ splits over a two-ended subgroup if and only if $\Delta$ contains a separating full subcomplex which is a cut pair or a suspension of a simplex. Since $\Delta\subset \field{S}^2$ is a connected flag complex with no separating vertex and no separating edge, it is clear that if $\Delta$ has a separating full subcomplex which is a suspension of a simplex then $\Delta$ has a cut pair or a separating induced path of length $2$.
\end{proof}

We now restrict our attention to right-angled Coxeter groups which are one-ended, non-hyperbolic, and not virtually $\Z^2$. Therefore, we need the following assumptions on defining nerves of right-angled Coxeter groups.

\begin{sassu}The planar flag complex $\Delta\subset \field{S}^2$:
\begin{enumerate}
    \item is connected with no separating vertices and no separating edges ($G_\Delta$ is one-ended);
    \item contains at least one induced $4$-cycle ($G_\Delta$ is not hyperbolic);
    \item is not a $4$-cycle and not a cone of a $4$-cycle ($G_\Delta$ is not virtually $\Z^2$).
\end{enumerate}
\end{sassu}

\subsection{Decomposition of planar flag complexes}

\label{tree1}

In \cite{NT}, the second and thirds authors describe a tree-like decomposition for triangle-free planar graphs. This decomposition has ``nice'' vertex graphs and is one of the key ideas of \cite{NT}. The techniques of \cite{NT} also apply to planar flag complexes. Starting with some terminology, we review the basics of this construction in the setting of planar flag complexes.

\begin{defn}
An induced 4--cycle $\sigma$ of a flag complex $\Delta$ \emph{separates $\Delta$} if $\Delta-\sigma$ has at least two components.
\end{defn}


\begin{defn}
\label{de1}
Let $\Delta\subset \field{S}^2$ be a planar flag complex. An induced $4$--cycle $\sigma$ \emph{strongly separates} $\Delta$ if $\Delta$ has non-empty intersection with both components of $\field{S}^2-\sigma$.
The complex $\Delta$ is called \emph{prime} if $\Delta$ satisfies the following conditions:
\begin{enumerate}
    \item $\Delta$ is connected with no separating vertex and no separating edge;
    \item $\Delta$ is not a $4$--cycle but contains at least one induced $4$-cycle;
    \item $\Delta$ has no strongly separating induced $4$-cycle (i.e. each induced $4$-cycle bounds a region of $\field{S}^2-\Delta$).
\end{enumerate}
\end{defn}

\begin{exmp}
The suspension of a non-triangle graph (not necessarily connected) with $3$ vertices is prime.
\end{exmp}

\begin{rem}
The notion of strongly separating in Definition~\ref{de1} depends on the choice of embedding map of the ambient flag complex into the sphere $\field{S}^2$. This notion is also based on the Jordan Curve Theorem that $\field{S}^2-\sigma$ has two components.
\end{rem}

\begin{figure}
\begin{tikzpicture}[scale=0.5]
\draw (0,0) node[circle,fill,inner sep=1.5pt, color=black](1){} -- (2,2) node[circle,fill,inner sep=1.5pt, color=black](1){}-- (4,0) node[circle,fill,inner sep=1.5pt, color=black](1){}-- (2,-2) node[circle,fill,inner sep=1.5pt, color=black](1){}-- (0,0) node[circle,fill,inner sep=1.5pt, color=black](1){};\draw (2,2) node[circle,fill,inner sep=1.5pt, color=black](1){} -- (2,0) node[circle,fill,inner sep=1.5pt, color=black](1){}-- (2,-2) node[circle,fill,inner sep=1.5pt, color=black](1){};

\filldraw[fill=red!40!white, draw=black] (6,0) node[circle,fill,inner sep=1.5pt, color=black](1){} -- (8,2) node[circle,fill,inner sep=1.5pt, color=black](1){} -- (8,0)node[circle,fill,inner sep=1.5pt, color=black](1){} -- (6,0)node[circle,fill,inner sep=1.5pt, color=black](1){};
\filldraw[fill=red!40!white, draw=black] (6,0) node[circle,fill,inner sep=1.5pt, color=black](1){} -- (8,-2) node[circle,fill,inner sep=1.5pt, color=black](1){} -- (8,0)node[circle,fill,inner sep=1.5pt, color=black](1){} -- (6,0)node[circle,fill,inner sep=1.5pt, color=black](1){};
\draw (8,2) node[circle,fill,inner sep=1.5pt, color=black](1){} -- (10,0) node[circle,fill,inner sep=1.5pt, color=black](1){}-- (8,-2) node[circle,fill,inner sep=1.5pt, color=black](1){};

\filldraw[fill=red!40!white, draw=black] (12,0) node[circle,fill,inner sep=1.5pt, color=black](1){} -- (14,2) node[circle,fill,inner sep=1.5pt, color=black](1){} -- (14,0)node[circle,fill,inner sep=1.5pt, color=black](1){} -- (12,0)node[circle,fill,inner sep=1.5pt, color=black](1){};
\filldraw[fill=red!40!white, draw=black] (12,0) node[circle,fill,inner sep=1.5pt, color=black](1){} -- (14,-2) node[circle,fill,inner sep=1.5pt, color=black](1){} -- (14,0)node[circle,fill,inner sep=1.5pt, color=black](1){} -- (12,0)node[circle,fill,inner sep=1.5pt, color=black](1){};
\filldraw[fill=red!40!white, draw=black] (14,0) node[circle,fill,inner sep=1.5pt, color=black](1){} -- (14,2) node[circle,fill,inner sep=1.5pt, color=black](1){} -- (16,0)node[circle,fill,inner sep=1.5pt, color=black](1){} -- (14,0)node[circle,fill,inner sep=1.5pt, color=black](1){};
\filldraw[fill=red!40!white, draw=black] (14,0) node[circle,fill,inner sep=1.5pt, color=black](1){} -- (14,-2) node[circle,fill,inner sep=1.5pt, color=black](1){} -- (16,0)node[circle,fill,inner sep=1.5pt, color=black](1){} -- (14,0)node[circle,fill,inner sep=1.5pt, color=black](1){};

\end{tikzpicture}
\caption{All possible special prime flag complexes}
\label{special}
\end{figure}
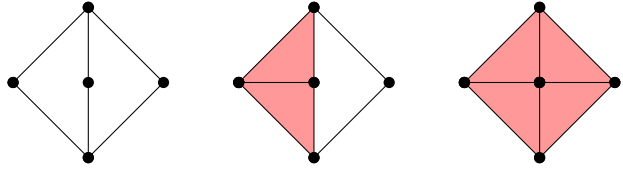

\begin{defn}
A prime flag complex $\Delta$ is called \emph{special} if it the suspension of a non-triangle graph (not necessarily connected) with 3 vertices. We remark that there are only three possible special prime flag complexes (see Figure~\ref{special}). 
\end{defn}

\begin{figure}
\begin{tikzpicture}[scale=0.5]

\draw (4,0) node[circle,fill,inner sep=1.5pt, color=black](1){} -- (2,3) node[circle,fill,inner sep=1.5pt, color=black](1){} -- (6.5,5) node[circle,fill,inner sep=1.5pt, color=black](1){}-- (11,3) node[circle,fill,inner sep=1.5pt, color=black](1){} -- (9,0) node[circle,fill,inner sep=1.5pt, color=black](1){} -- (4,0) node[circle,fill,inner sep=1.5pt, color=black](1){};

\draw (5,1) node[circle,fill,inner sep=1.5pt, color=black](1){} -- (3.8,2.8) node[circle,fill,inner sep=1.5pt, color=black](1){} -- (6.5,3.6) node[circle,fill,inner sep=1.5pt, color=black](1){}-- (9.2,2.8) node[circle,fill,inner sep=1.5pt, color=black](1){} -- (8,1) node[circle,fill,inner sep=1.5pt, color=black](1){} -- (5,1) node[circle,fill,inner sep=1.5pt, color=black](1){};

\draw (4,0) node[circle,fill,inner sep=1.5pt, color=black](1){} -- (5,1) node[circle,fill,inner sep=1.5pt, color=black](1){};
\draw (2,3) node[circle,fill,inner sep=1.5pt, color=black](1){} -- (3.8,2.8) node[circle,fill,inner sep=1.5pt, color=black](1){};
\draw (6.5,5) node[circle,fill,inner sep=1.5pt, color=black](1){} -- (6.5,3.6) node[circle,fill,inner sep=1.5pt, color=black](1){};
\draw (11,3) node[circle,fill,inner sep=1.5pt, color=black](1){} --  (9.2,2.8) node[circle,fill,inner sep=1.5pt, color=black](1){};
\draw (9,0) node[circle,fill,inner sep=1.5pt, color=black](1){} --  (8,1) node[circle,fill,inner sep=1.5pt, color=black](1){};

\filldraw[fill=red!40!white, draw=black] (6.5,3.6) node[circle,fill,inner sep=1.5pt, color=black](1){} -- (9.2,2.8) node[circle,fill,inner sep=1.5pt, color=black](1){} -- (8,1)node[circle,fill,inner sep=1.5pt, color=black](1){} -- (6.5,3.6)node[circle,fill,inner sep=1.5pt, color=black](1){};

\filldraw[fill=red!40!white, draw=black] (6.5,5) node[circle,fill,inner sep=1.5pt, color=black](1){} -- (2,3) node[circle,fill,inner sep=1.5pt, color=black](1){} -- (3.8,2.8)node[circle,fill,inner sep=1.5pt, color=black](1){} -- (6.5,5)node[circle,fill,inner sep=1.5pt, color=black](1){};

\filldraw[fill=red!40!white, draw=black] (6.5,5) node[circle,fill,inner sep=1.5pt, color=black](1){} -- (6.5,3.6) node[circle,fill,inner sep=1.5pt, color=black](1){} -- (3.8,2.8)node[circle,fill,inner sep=1.5pt, color=black](1){} -- (6.5,5)node[circle,fill,inner sep=1.5pt, color=black](1){};

\filldraw[fill=red!40!white, draw=black] (11,3) node[circle,fill,inner sep=1.5pt, color=black](1){} -- (9,0) node[circle,fill,inner sep=1.5pt, color=black](1){} -- (9.2,2.8)node[circle,fill,inner sep=1.5pt, color=black](1){} -- (11,3)node[circle,fill,inner sep=1.5pt, color=black](1){};

\filldraw[fill=red!40!white, draw=black] (8,1) node[circle,fill,inner sep=1.5pt, color=black](1){} -- (9,0) node[circle,fill,inner sep=1.5pt, color=black](1){} -- (9.2,2.8)node[circle,fill,inner sep=1.5pt, color=black](1){} -- (8,1)node[circle,fill,inner sep=1.5pt, color=black](1){};

\draw (15,0) node[circle,fill,inner sep=1.5pt, color=black](1){} -- (20,0) node[circle,fill,inner sep=1.5pt, color=black](1){} -- (20,5) node[circle,fill,inner sep=1.5pt, color=black](1){}-- (15,5) node[circle,fill,inner sep=1.5pt, color=black](1){} -- (15,0) node[circle,fill,inner sep=1.5pt, color=black](1){};

\draw (16,1) node[circle,fill,inner sep=1.5pt, color=black](1){} -- (19,1) node[circle,fill,inner sep=1.5pt, color=black](1){} -- (19,4) node[circle,fill,inner sep=1.5pt, color=black](1){}-- (16,4) node[circle,fill,inner sep=1.5pt, color=black](1){} -- (16,1) node[circle,fill,inner sep=1.5pt, color=black](1){};

\draw (15,0) node[circle,fill,inner sep=1.5pt, color=black](1){} -- (16,1) node[circle,fill,inner sep=1.5pt, color=black](1){};
\draw (20,0) node[circle,fill,inner sep=1.5pt, color=black](1){} -- (19,1) node[circle,fill,inner sep=1.5pt, color=black](1){};
\draw (20,5) node[circle,fill,inner sep=1.5pt, color=black](1){} -- (19,4) node[circle,fill,inner sep=1.5pt, color=black](1){};
\draw (15,5) node[circle,fill,inner sep=1.5pt, color=black](1){} --  (16,4) node[circle,fill,inner sep=1.5pt, color=black](1){};

\filldraw[fill=red!40!white, draw=black] (16,1) node[circle,fill,inner sep=1.5pt, color=black](1){} -- (19,1) node[circle,fill,inner sep=1.5pt, color=black](1){} -- (19,4)node[circle,fill,inner sep=1.5pt, color=black](1){} -- (16,1)node[circle,fill,inner sep=1.5pt, color=black](1){};

\filldraw[fill=red!40!white, draw=black] (16,1) node[circle,fill,inner sep=1.5pt, color=black](1){} -- (16,4) node[circle,fill,inner sep=1.5pt, color=black](1){} -- (19,4)node[circle,fill,inner sep=1.5pt, color=black](1){} -- (16,1)node[circle,fill,inner sep=1.5pt, color=black](1){};

\filldraw[fill=red!40!white, draw=black] (19,1) node[circle,fill,inner sep=1.5pt, color=black](1){} -- (19,4) node[circle,fill,inner sep=1.5pt, color=black](1){} -- (20,5)node[circle,fill,inner sep=1.5pt, color=black](1){} -- (19,1)node[circle,fill,inner sep=1.5pt, color=black](1){};

\filldraw[fill=red!40!white, draw=black] (19,1) node[circle,fill,inner sep=1.5pt, color=black](1){} -- (20,0) node[circle,fill,inner sep=1.5pt, color=black](1){} -- (20,5)node[circle,fill,inner sep=1.5pt, color=black](1){} -- (19,1)node[circle,fill,inner sep=1.5pt, color=black](1){};
\end{tikzpicture}

\caption{Examples of non-special prime flag complexes}
\label{non-special}
\end{figure}
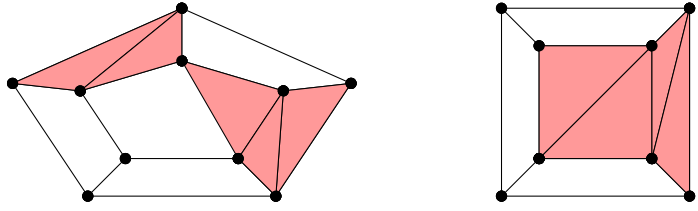

The following lemma will help us understand the structure of prime flag complexes and it can be compared to Lemma 3.7 in \cite{NT}.

\begin{lem}
\label{le0}
Let $\Delta\subset \field{S}^2$ be a prime flag complex. If $\Delta$ contains a full subcomplex $B$ which is a special prime flag complex, then $\Delta$ is exactly the complex $B$. In particular, the 1-skeleton of $\Delta$ is $\mathcal{CFS}$ if and only if $\Delta$ is special.
\end{lem}

\begin{proof}
Assume by way of contradiction that $\Delta$ is not equal $B$. Then there is a vertex of $v$ of $\Delta$ which is not a vertex of $B$. Therefore, $v$ is a point in $\field{S}^2-B$. Since $B$ is a special prime flag complex, $B$ is one of the complexes in Figure~\ref{special}. By working on each case of $B$ in Figure~\ref{special}, we always can find an induced $4$-cycle $\sigma$ in $B$ such that $v$ and the unique vertex $u$ of $B-\sigma$ lies in different components of $\field{S}^2-\sigma$. This implies that $\sigma$ strongly separates $\Delta$. Therefore, $\Delta$ is not prime which is a contradiction.

We now prove that the 1-skeleton of $\Delta$ is $\mathcal{CFS}$ if and only if $\Delta$ is special. If $\Delta$ is special, we can see from Figure~\ref{special} that the 1-skeleton of $\Delta$ is $\mathcal{CFS}$. We now assume that the 1-skeleton of $\Delta$ is $\mathcal{CFS}$. Then $\Delta$ contains a full subcomplex $B$ which is a special prime flag complex. Therefore, $\Delta$ must be exactly the complex $B$ as we proved above. Thus, $\Delta$ is special.
\end{proof}

The following lemma will play an important role in understanding the relatively hyperbolic structure of right-angled Coxeter groups of planar defining nerves (see Theorem~\ref{sosonice}).

\begin{lem}
\label{for rel}
Let $\Delta\subset \field{S}^2$ be a non-special prime flag complex. Let $\alpha_1$ and $\alpha_2$ be two distinct induced 4-cycles of $\Delta$ which have non-empty intersection. Then $\alpha_1\cap\alpha_2$ is a vertex or an edge of $\Delta$. 
\end{lem}

\begin{proof}
Assume by way of contradiction that $\alpha_1\cap\alpha_2$ contains two non-adjacent vertices, then the subcomplex of $\Delta$ induced by $\alpha_1\cup\alpha_2$ contains a special full subcomplex $B$. In particular, $\Delta$ contains the special full subcomplex $B$. Also $\Delta$ is a prime flag complex. Therefore, by Lemma~\ref{le0}, $\Delta=B$ is special which is a contradiction. Therefore, $\alpha_1\cap\alpha_2$ is a vertex or an edge of $\Delta$. 
\end{proof}

The following definition is an extension of the concept of strong visual decomposition of triangle free planar graphs (see Definition~3.4 in \cite{NT}) to planar flag complexes.

\begin{defn}
Let $\Delta\subset \field{S}^2$ be a planar flag complex and $\sigma$ a strongly separating $4$-cycle of $\Delta$. Then $\field{S}^2-\sigma$ has two components $U_1$ and $U_2$ which both intersect $\Delta$. For each $i=1,2$, let $\Delta_i$ be $\sigma$ together with components of $\Delta-\sigma$ in $U_i$. Then, $\Delta=\Delta_1\cup\Delta_2$ and $\Delta_1\cap\Delta_2=\sigma$. We call the pair $(\Delta_1,\Delta_2)$ a \emph{strong visual decomposition} of $\Delta$ with respect to the given embedding of $\Delta$ into $\field{S}^2$. If the embedding is clear from the context, we just say the pair $(\Delta_1,\Delta_2)$ is a strong visual decomposition of $\Delta$ along $\sigma$.
\end{defn}

The proof of the following proposition is analogous to the proof of Proposition~3.11 in \cite{NT}. We leave the details to the reader.

\begin{prop}
\label{prop:keyidea1}
Let $\Delta\subset \field{S}^2$ be a planar flag complex satisfying Standing Assumptions. Then there is a finite tree $T$ that encodes the structure of $\Delta$ as follows:
\begin{enumerate}
\item Each vertex $v$ of $T$ is associated to a full subcomplex $\Delta_v$ of $\Delta$ which is prime. Moreover, $\Delta_v\neq\Delta_{v'}$ if $v\neq v'$ and $\bigcup_{v\in V(T)}^{}\Delta_v=\Delta$.
\item Each edge $e$ of $T$ is associated to an induced 4--cycle $\Delta_e$ of $\Delta$. Moreover, $\Delta_e\neq\Delta_{e'}$ if $e\neq e'$.

\item Two vertices $v_1$ and $v_2$ of $T$ are endpoints of the same edge $e$ if and only if $\Delta_{v_1}\cap\Delta_{v_2}=\Delta_e$. Moreover, if $V_1$ and $V_2$ are vertex sets of two components of $T$ removed the midpoint of $e$, then $(\bigcup_{v\in V_1}\Delta_v,\bigcup_{v\in V_2}\Delta_v)$ is a strong visual decomposition of $\Delta$ along $\Delta_e$.
\end{enumerate}
Moreover, the 1-skeleton of $\Delta$ is $\mathcal{CFS}$ if and only if the 1-skeleton of $\Delta_v$ is also $\mathcal{CFS}$ for each vertex $v$ of $T$ (i.e. $\Delta_v$ is a special prime complex by Lemma \ref{le0}).
\end{prop}

\subsection{Relatively hyperbolic structure and manifold structure of RACGs with planar defining nerves}
\label{rhsms}

In this subsection, we are going to discuss the proof of Theorem~\ref{mani}, Theorem~\ref{sosonice}, and Theorem~\ref{SC-carpet}. Before proving Theorem~\ref{mani}, we discuss some key concepts in the following remark.




\begin{rem}
\label{imp}
Let $\Delta\subset \field{S}^2$ be a planar flag complex satisfying Standing Assumptions. Then the divergence of the right-angled Coxeter group $G_\Delta$ is linear if and only if the $1$--skeleton of $\Delta$ is a join of two graphs of diameters at least $2$ (see Theorem~\ref{dt}). Since $\Delta\subset \field{S}^2$ is a planar flag complex satisfying Standing Assumptions, the $1$--skeleton of $\Delta$ is a join of two graphs of diameters at least $2$ if and only if $\Delta$ is a suspension of a graph $K$ where $K$ is an $n$--cycle for $n\geq 4$ or $K$ has at least $3$ vertices and it is a finite disjoint union of vertices and finite trees with vertex degrees $1$ or $2$. If $K$ has at least $3$ vertices and it is a finite disjoint union of vertices and finite trees with vertex degrees $1$ or $2$, then we call $K$ a \emph{broken line}. 
\end{rem}

We begin with a few words about the proof of Theorem~\ref{sosonice}. The proof of Theorem~\ref{sosonice} is completely analogous to that of Theorem 1.6 of \cite{NT}, and we leave the details to the reader. The key ideas behind the proof of Theorem 1.6 of \cite{NT} are the tree structure of graphs and the fact that the intersection of two induced $4$--cycles in a prime graph which is not a suspension of three points is empty, a vertex, or an edge. In the case of planar flag complexes, the analogous facts are addressed by Proposition~\ref{prop:keyidea1} and Lemma~\ref{for rel}. We will use Theorem~\ref{sosonice} to prove Theorem~\ref{mani}.

\begin{proof}[Proof of Theorem~\ref{mani}]
We first show that $G_\Delta$ is virtually the fundamental group of a 3-manifold $M$ with empty or toroidal boundary if and only if the boundary of each region, if it exists, in $\field{S}^2-\Delta$ is a 4-cycle.

We are now going to prove the necessity. Suppose that the boundary of each region, if it exists, in $\field{S}^2-\Delta$ is a 4-cycle. Let $\Gamma$ be the $1$-skeleton of $\Delta$. Then the boundary of each region in $\field{S}^2-\Gamma$ is a $3$-cycle or a 4-cycle. The group $G_\Delta$ acts by reflections on a simply-connected 3-manifold $N$ with fundamental domain a ball whose boundary is the cell structure on $\field{S}^2$ that is dual to $\Gamma$. Since the boundary of each region in $\field{S}^2-\Gamma$ is a $3$-cycle or a 4-cycle, the stabilizer of each vertex is $\Z_2^3$ or $D_\infty\times D_\infty$. Let $H$ be a torsion-free finite index subgroup of $G_\Delta$. The quotient space $N/H$ has the property that the link of each vertex associated to a region with $4$-cycle boundary (if it exists) is a torus, and we can thus remove a finite neighborhood from such a vertex (if it exists) to obtain a desired manifold $M$ whose boundary is empty or a union of tori.

We now prove sufficiency. Assume that $G_\Delta$ has a finite index subgroup $H$ such that $H$ is the fundamental group of a 3-manifold $M$ with empty or toroidal boundary. Then the boundary of each region in $\field{S}^2-\Gamma$ must be a $3$-cycle or a 4-cycle, otherwise the Euler characteristic $\chi(H)$ is negative which is a contradiction (see Section 3.b in \cite{KSD}). Therefore, the boundary of each region, if it exists, in $\field{S}^2-\Delta$ is a 4-cycle. 

Next, we consider all possible types of $3$--manifold $M$ via graph theoretic properties on $\Delta$. 
The fact (\ref{item:seifert}) is clearly true since a right-angled Coxeter group induced by an $n$-cycle ($n\geq 4$) or a broken line is virtually a surface group. For the fact (\ref{item: graphmld}) we first observe that the 1-skeleton of $\Delta$ is not a join of two graphs of diameter at least 2. 
Therefore, by Theorem~\ref{dt} the divergence of $G_\Delta$ (also the divergence of $\pi_1(M)$) is quadratic. By Theorem~\ref{thm:GKL:divergence} $M$ must be a graph manifold. 

For the facts (\ref{item:mixed}) and (\ref{item:hyper}) we note that if the $1$-skeleton of $\Delta$ is not $\mathcal{CFS}$, then then the divergence of $G_\Delta$ (also the divergence of $\pi_1(M)$) is exponential by Theorem~\ref{sosonice}. Therefore, $M$ must be a hyperbolic manifold with boundary or a mixed manifold. If, in addition, $\Delta$ contains at least a separating induced $4$-cycle, then $M$ contains at least a JSJ torus and it must be a mixed manifold which proves (\ref{item:mixed}).

For (\ref{item:hyper}) we see that $\Delta$ is non-special prime flag complex. Therefore, in this case, $G_\Delta$ is hyperbolic relative to the collection of all right-angled Coxeter subgroups $G_\sigma$ where $\sigma$ is an induced $4$-cycles of $\Delta$. Due to the peripheral structure of $G_\Delta$, the JSJ decomposition of $M$ cannot have a Seifert piece. Moreover, the JSJ decomposition of $M$ must consist of a single piece, otherwise $\pi_1(M)$ would split over $\Z^2$. 
This implies that $M$ is a hyperbolic manifold with boundary.   \end{proof}

We now study the relatively hyperbolic structure of right-angled Coxeter groups which are not virtually the fundamental group of a 3-manifold with empty or toroidal boundary. We prove that each such group is hyperbolic relative to a collection of its proper subgroups and we then study its Bowditch boundary with respect to the minimal peripheral structure. The following two lemmas will help us study the relatively hyperbolic structure of right-angled Coxeter groups which are not virtually the fundamental group of a 3-manifold with empty or toroidal boundary.

\begin{lem}
\label{tt}
Let $\Delta\subset \field{S}^2$ be a planar flag complex. Let $(\Delta_1,\Delta_2)$ be a strong visual decomposition of $\Delta$ along some induced $4$--cycle $\sigma$. If $R$ is a region of $\field{S}^2-\Delta$, then $R$ is a region of either $\field{S}^2-\Delta_1$ or $\field{S}^2-\Delta_2$. 
\end{lem}

\begin{proof}
Let $U$ and $V$ be two regions of $\field{S}^2-\sigma$ such that $\Delta_1-\sigma\subset U$ and $\Delta_2-\sigma\subset V$. Since $R$ is a connected subset in $\field{S}^2-\sigma$, $R$ lies inside either $U$ or $V$ (say $U$). We now prove that $R$ is also a region of $\field{S}^2-\Delta_1$. We first observe that $\field{S}^2-\Delta\subset \field{S}^2-\Delta_1\subset (\field{S}^2-\Delta)\cup V$. Therefore $R$ lies in some region $R_1$ of $\field{S}^2-\Delta_1$. Since $R_1$ is a connected subset of $\field{S}^2-\sigma$,  $R_1$ lies either in $U$ or $V$, say $U$. Moreover, $R_1\cap U$ contains the non-empty set $R$. Therefore $R_1$ must lie in $U$ which implies that $R_1\cap V=\emptyset$. Thus $R_1$ is a connected subset of $\field{S}^2-\Delta$. This implies that $R_1$ must lie in some region of $\field{S}^2-\Delta$ and the region must be $R$. Therefore, $R=R_1$ is also a region of $\field{S}^2-\Delta_1$.
\end{proof}

\begin{lem}
\label{aa}
Let $\Delta\subset \field{S}^2$ be a planar flag complex satisfying Standing Assumptions. Let $T$ be a finite tree that encodes the structure of $\Delta$ as in Proposition~\ref{prop:keyidea1}. Then for each region $R$ of $\field{S}^2-\Delta$ there is a vertex $v$ of $T$ such that $R$ is also a region of $\field{S}^2-\Delta_v$. 
\end{lem}

\begin{proof}
The lemma can be proved by induction on the number of vertices of tree $T$ by using Lemma~\ref{tt}.
\end{proof}

In the rest of this subsection, we prove Theorem~\ref{SC-carpet}. The following proposition shows the existence of a non-trivial relatively hyperbolic structure of right-angled Coxeter groups which are not virtually the fundamental group of a 3-manifold with empty or toroidal boundary.

\begin{prop}
\label{pppp1}
Let $\Delta\subset \field{S}^2$ be a connected flag complex with no separating vertex and no separating edge. Assume that the boundary of some region in $\field{S}^2-\Delta$ is an $n$-cycle for some $n\geq 5$. Then the $1$-skeleton of $\Delta$ is not $\mathcal{CFS}$. In particular, the divergence of $G_\Delta$ is exponential (i.e. $G_\Delta$ is non-trivially relatively hyperbolic). 
\end{prop}

\begin{proof}
Since the boundary of some region in $\field{S}^2-\Delta$ is an $n$-cycle for some $n\geq 5$, it follows that $\Delta$ is not a simplex, a $4$-cycle, or cone on a $4$-cycle. If $\Delta$ does not contain an induced $4$-cycle, then it is clear that the $1$-skeleton of $\Delta$ is not $\mathcal{CFS}$. Therefore, we assume that $\Delta$ contains at least one induced $4$-cycle. In this case, we note that $\Delta$ satisfies Standing Assumption.

Let $T$ be a tree that encodes the structure of $\Delta$ as in Proposition~\ref{prop:keyidea1}. By Lemma~\ref{aa} there is a vertex $v$ of $T$ such that the boundary of some region in $\field{S}^2-\Delta_v$ is an $n$-cycle for some $n\geq 5$. This implies that $\Delta_v$ is not a special prime complex. Therefore, by Proposition~\ref{prop:keyidea1} the $1$-skeleton of $\Delta$ is not $\mathcal{CFS}$. By Theorem~\ref{sosonice} the divergence of right-angled Coxeter group $G_\Delta$ is exponential.
\end{proof}

We now study the Bowditch boundary of the relatively hyperbolic right-angled Coxeter groups in Proposition~\ref{pppp1}. We begin with the definition of Sierpinski carpet and recall a result due to \'Swi\c{a}tkowski \cite{SJ} which gives sufficient conditions for when the $\CAT(0)$ boundary of a right-angled Coxeter group satisfying Standing Assumptions is a Sierpinski carpet.

\begin{defn}
Let $D_1, D_2,\cdots $ be a sequence of open disks in a 2-dimensional sphere $\field{S}^2$ such that
\begin{enumerate}
\item $\bar{D}_i \cap \bar{D}_j =\emptyset$ for $i\neq j$,
\item $diam(D_i) \to 0$ with respect to the round metric on $\field{S}^2$, and
\item $\bigcup D_i$ is dense. 
\end{enumerate}
Then $X= \field{S}^2-\bigcup D_i$ is a \emph{Sierpinski carpet}. The circles $C=\partial \bar{D}_i \subset X$ are called \emph{peripheral circles}.
\end{defn}

\begin{thm}[Theorem 1.3 in \cite{SJ}]
\label{basic}
Let $\Delta\subset \field{S}^2$ be a planar flag complex satisfying Standing Assumptions. Suppose also that the following two conditions hold:
\begin{enumerate}
    \item For any two distinct regions of $\field{S}^2-\Delta$, the intersection of their boundaries is empty, or a vertex, or an edge;
    \item The boundary $\sigma$ of each region of $\field{S}^2-\Delta$ satisfies the condition that for any two vertices of $\sigma$ staying at distance $2$ in $\Delta$ any induced path of length $2$ in the 1-skeleton of $\Delta$ connecting these vertices is entirely contained in $\sigma$.
\end{enumerate}
Then the $\CAT(0)$ boundary $\partial \Sigma_\Delta$ is the Sierpinski carpet.
\end{thm}



We now give a proof of Theorem~\ref{SC-carpet}. 

\begin{proof}[Proof of Theorem~\ref{SC-carpet}]
We assume that $\Delta$ satisfies both Conditions (1) and (2) of Theorem~\ref{SC-carpet} and we will prove that $G_\Delta$ has a non-trivial minimal peripheral structure with Bowditch boundary the Sierpinski carpet. We first show that the $\CAT(0)$ boundary $\partial \Sigma_\Delta$ is the Sierpinski carpet by using Theorem~\ref{basic}. Then we will use Theorem~\ref{cool} to prove that $G_\Delta$ has a non-trivial minimal peripheral structure with Bowditch boundary the Sierpinski carpet. Since the $1$-skeleton of $\Delta$ has no cut pair, $\Delta$ satisfies Condition (1) of Theorem~\ref{basic}. Also, the $1$-skeleton of $\Delta$ has no separating induced path of length $2$. 

Now, we show that $\Delta$ also satisfies Condition (2) of Theorem~\ref{basic}. By the way of contradiction that the boundary $\sigma$ of some region of $\field{S}^2-\Delta$ does not satisfy Condition (2). Then there is an induced path $\alpha$ of length $2$ in the $1$--skeleton of $\Delta$ such that $\sigma\cap \alpha$ consists of two distinct vertices $u$ and $v$. We now claim that $\alpha$ separates the $1$--skeleton of $\Delta$ which leads to a contradiction. We first assume that $\Delta=\alpha\cup \sigma$. Then $\Delta-\alpha=\sigma-\{u,v\}$ is disconnected. Therefore, $\alpha$ separates $\Delta$ which is also its $1$--skeleton. We now assume that there is a point $w$ in $\Delta-(\alpha\cup \sigma)$. We observe that $\sigma$ is the union of two paths $\sigma_1$ and $\sigma_2$ such that $\sigma_1\cap \sigma_2=\{u,v\}$. Since $\sigma$ bounds a region $R$ in $\field{S}^2-\Delta$, the path $\alpha$ lies completely outside the region $R$. Therefore, $\field{S}^2-(\alpha\cup \sigma)$ contains exactly three regions and $R$ is one of them. Since $R$ is also a region in $\field{S}^2-\Delta$, the point $w$ can not lie in $R$. Therefore, $w$ must lie in the region of $\field{S}^2-(\alpha\cup \sigma)$ which is bounded by $\alpha\cup \sigma_1$ or $\alpha\cup \sigma_2$ (say $\alpha\cup \sigma_1$). Therefore, $\alpha$ separates $w$ from a vertex of $\sigma_2$ in the $1$--skeleton of $\Delta$. This implies that $\alpha$ separates the $1$--skeleton of $\Delta$ which leads to a contradiction. Thus, $\Delta$ satisfies Condition (2) of Theorem~\ref{basic}. Therefore, the $\CAT(0)$ boundary $\partial \Sigma_\Delta$ is the Sierpinski carpet. 

Since the $1$-skeleton of $\Delta$ has no separating induced $4$-cycle, each induced $4$-cycle of $\Delta$ must bound a region of $\field{S}^2-\Delta$. Therefore, the intersection of two induced $4$-cycles of $\Delta$ is empty, or a vertex, or an edge. Thus by Theorem \ref{th1} the group $G_\Delta$ is relatively hyperbolic with respect to the collection $\PP$ of all right-angled Coxeter groups induced by some induced $4$-cycle of $\Delta$. We now prove the Bowditch boundary $\partial (G_\Gamma,\PP)$ is a Sierpinski carpet. 

We first claim that if $\sigma$ is the boundary of a region in $\field{S}^2-\Delta$, then the limit set of subgroup $G_\sigma$ is a peripheral circle of the Sierpinski carpet $\partial \Sigma_\Gamma$. We see that the intersection of $\sigma$ with an induced $4$-cycle of $\Delta$ is empty, or a vertex, or an edge. Therefore by Theorem \ref{th1} again $G_\Delta$ is relatively hyperbolic with respect to the collection $\bar{\PP}=\PP\cup \{G_\sigma\}$. By Theorem \ref{cool}, the Bowditch boundary $\partial (G_\Gamma,\bar{\PP})$ is obtained from the $\CAT(0)$ boundary $\partial \Sigma_\Gamma$ by identifying the limit set of each peripheral left coset of a subgroup in $\bar{\PP}$ to a point. Let $f$ be this quotient map. Let $v_{G_K}$ be the point in $\partial (G_\Gamma,\bar{\PP})$ that is the image of the limit set of subgroup $G_\sigma$ under the map $f$. Suppose by way of contradiction that the limit set of the subgroup $G_{\sigma}$ is a separating circle of the Sierpinski carpet $\partial\Sigma_{\Gamma}$. Then the point $v_{G_K}$ is a global cut point of $\partial (G_\Gamma,\bar{\PP})$. We know that $\sigma$ bounds a region of $\field{S}^2-\Delta$ while the intersection of the boundaries of two distinct regions of $\field{S}^2-\Delta$ is empty, or a vertex, or an edge. This implies that no induced subgraphs of $\sigma$ separates the $1$-skeleton of $\Delta$. Therefore by Theorem \ref{intro2} point $v_{G_K}$ is not a global cut point of $\partial (G_\Gamma,\bar{\PP})$ which is a contradiction. Thus the limit set of subgroup $G_\sigma$ is a peripheral circle of the Sierpinski carpet $\partial \Sigma_\Gamma$.

By Theorem \ref{cool} again, the Bowditch boundary $\partial (G_\Gamma,\PP)$ is obtained from the Sierpinski carpet $\partial \Sigma_\Gamma$ by identifying each peripheral circle of $\partial \Sigma_\Gamma$ which is the limit set of a peripheral left coset of a subgroup in $\PP$ to a point. Let $h$ be this quotient map. Then $h$ is $G_\Gamma$--equivariant. By Condition (1) some boundary $\gamma$ of a region of $\field{S}^2-\Delta$ is an $n$-cycle with $n\geq 5$. Therefore, the limit set of subgroup $G_\gamma$ is a peripheral circle $C$ of the Sierpinski carpet $\partial \Sigma_\Gamma$ by the above argument. Since $G_\gamma$ is not a group in $\PP$, the peripheral circle $C$ is not collapsed to a point via the map $h$. 

We now prove that the Bowditch boundary $\partial (G_\Gamma,\PP)$ is a Sierpinski carpet. Let $\mathcal{L}$ be the collection of all translates of the peripheral circle $C$ by group elements in $G_\Delta$. Then all peripheral circles in $\mathcal{L}$ survive via the map $h$. Therefore, we can consider $\mathcal{L}$ as a collection of pairwise disjoint circles in $\partial (G_\Gamma,\PP)$. Moreover, $\mathcal{L}$ is a $G_\Gamma$-invariant collection in $\partial (G_\Gamma,\PP)$ since the quotient map $h$ is $G_\Gamma$-equivariant. Also the action of $G_\Gamma$ on $\partial (G_\Gamma,\PP)$ is minimal (i.e. the orbit of each single point is dense in $\partial (G_\Gamma,\PP)$) by \cite {MR2922380}. Therefore, the union of all circles in $\mathcal{L}$ is dense in $\partial (G_\Gamma,\PP)$. Fix metrics on $\CAT(0)$ boundary $\partial \Sigma_\Gamma$ and Bowditch boundary $\partial (G_\Gamma,\PP)$. We consider $\mathcal{L}$ as a sequence $(C_n)$ of circles and we need to prove that $\diam(C_n)\to 0$ with respect to the metric on $\partial (G_\Gamma,\PP)$. 

Let $\epsilon$ be an arbitrary positive number and let $\{U_\alpha\}_{\alpha \in \Lambda}$ be a cover of $\partial (G_\Gamma,\PP)$ consisting of the open ball with diameter $\epsilon$. Then $\{h^{-1}(U_\alpha)\}_{\alpha \in \Lambda}$ is an open cover of $\partial \Sigma_\Gamma$. Since $\diam(C_n)\to 0$ with respect to the metric on $\partial \Sigma_\Gamma$ and $\partial \Sigma_\Gamma$ is a compact space, then each set $C_n$ lies in some member of the cover $\{h^{-1}(U_\alpha)\}_{\alpha \in \Lambda}$ for each $n$ sufficiently large. Therefore, each set $C_n=h(C_n)$ also lies in some member of the cover $\{U_\alpha\}_{\alpha \in \Lambda}$ for each $n$ sufficiently large. This implies that the diameter of each such circle $C_n$ is less than $\epsilon$ with respect to the metric on $\partial (G_\Gamma,\PP)$. This implies that $\diam(C_n)\to 0$ with respect to the metric on $\partial (G_\Gamma,\PP)$. Therefore, the Bowditch boundary $\partial (G_\Gamma,\PP)$ is a Sierpinski carpet by \cite{MR0099638}.

We now assume that $G_\Delta$ has a non-trivial minimal peripheral structure with Bowditch boundary the Sierpinski carpet. We will prove that $\Delta$ satisfies both Conditions (1) and (2). In fact, if $\Delta$ does not satisfy Condition (1), then $G_\Gamma$ is virtually a $3$-manifold group with empty boundary or tori boundary by Theorem~\ref{mani}. 
Therefore, the Bowditch boundary of $G_\Delta$ with respect to a non-trivial minimal peripheral structure is not the Sierpinski carpet which is a contradiction. Therefore, $\Delta$ must satisfy Condition (1). Also $\Delta$ must satisfy Condition (2). Otherwise by Theorem~\ref{intro2} and Theorem~\ref{intro3} the Bowditch boundary of $G_\Delta$ with respect to a non-trivial minimal peripheral structure contains a global cut point or a cut pair which is a contradiction.

\end{proof}

\subsection{Quasi-isometry classification of virtually graph manifold group RACGs}
\label{tree2}

In this subsection, we prove Theorem~\ref{thm:graphmanifold}. This theorem generalizes Theorem~1.1 in \cite{NT} by removing the condition ``triangle-free'' from the hypotheses.

In the rest of this section, we will assume that the planar flag complex $\Delta\subset \field{S}^2$ satisfies Standing Assumptions and the $1$-skeleton of $\Delta$ is a non-join $\mathcal{CFS}$ graph. In particular, $\Delta$ is not the suspension of an $n$-cycle. Moreover, $\Delta$ also does not contain a proper full subcomplex which is the suspension of an $n$-cycle since $\Delta\subset \field{S}^2$ is a planar flag complex. Let $T$ be a tree that encodes the structure of $\Delta$ as in Proposition~\ref{prop:keyidea1}. Since the $1$--skeleton of $\Delta$ is $\mathcal{CFS}$, it is shown in Proposition~\ref{prop:keyidea1} that each vertex subcomplex $\Delta_v$ is the suspension of a non-triangle graph with $3$ vertices. For the purpose of obtaining the quasi-isometry classification of our $\mathcal{CFS}$ right-angled Coxeter groups, the tree structure $T$ in Proposition~\ref{prop:keyidea1} is not the correct tree structure to use. So, we now modify the tree $T$ to obtain a new two-colored new tree that encodes the structure of $\Delta$.

\begin{cons}
\label{cons:keyidea2}
Step 1: We color an edge of $T$ by one of two colors: red and blue as follows. Let $e$ be an edge of $T$ with vertices $v_1$ and $v_2$. If either $\Delta_{v_1}$ or $\Delta_{v_2}$ is a suspension of a path of length $2$ we color the edge $e$ red. Next, we consider two cases. If $\Delta_{v_1}$ and $\Delta_{v_2}$ have the same suspension points, then we color the edge $e$ red. Otherwise, we color $e$ blue.

Step 2: Let $\mathcal{R}$ be the union of all red edges of $T$. We remark that $\mathcal{R}$ is not necessarily connected. We form a new tree $T_{r}$ from the tree $T$ by collapsing each component $C$ of $\mathcal{R}$ to a vertex labelled by $v_C$ and we associate each such new vertex $v_C$ to the complex $\Delta_{v_C}=\bigcup_{v\in V(C)} \Delta_v$. For each vertex $v$ of $T_r$ which is also a vertex of $T$ we still assign $v$ the complex $\Delta_v$ as in the previous tree $T$ structure. We observe that for each vertex $v$ of the new tree $T_r$ the associated vertex complex $\Delta_{v}$ is either one of the first two complexes in Figure~\ref{special} or a union of at least two special complexes which all share a pair of suspension vertices. Therefore, $\Delta_{v}$ is the suspension of a broken line $\ell_{v}$ which is distinct from a path of length $2$. We call the number of induced $4$-cycles of $\Delta_v\subset \field{S}^2$ that bound a region in $\field{S}^2-\Delta_v$ the \emph{weight} of $v$ denoted by $w(v)$. 
By construction, each edge $e$ of the new tree $T_r$ is also an edge of the old tree $T$. Therefore, we still assign the edge $e$ in the new tree $T_r$ the complex $\Delta_v$ as in the previous tree $T$ structure. 

We observe that for each edge $e$ of $T_r$ with vertices $v_1$ and $v_2$ the induced $4$-cycle $\Delta_e$ bounds a region in both $\field{S}^2-\Delta_{v_1}$ and $\field{S}^2-\Delta_{v_2}$. Therefore, the weight of each vertex $v$ is always greater than or equal to the degree of $v$ in $T_r$. Moreover, if $v_1$ and $v_2$ are two adjacent vertices in $T_r$, then suspension vertices of $\Delta_{v_1}$ are vertices of $\ell_{v_2}$ and similarly suspension vertices of $\Delta_{v_2}$ are vertices of $\ell_{v_1}$.


Step 3: We now color vertices of $T_r$. For each vertex $v$ of $T_r$, the complex $\Delta_v$ is a suspension of a broken line set $\ell_v$. Observe that in Step 2 that the weight of each vertex $v$ is always greater than or equal to the degree of $v$ in $T$. Therefore, we now color $v$ black if its weight is strictly greater than its degree. Otherwise, we color $v$ white.
\end{cons}

We now summarize some key properties of the tree $T_{r}$.

\begin{enumerate}
\item Each vertex $v$ of $T_r$ is associated to a full subcomplex $\Delta_v$ of $\Delta$ that is a suspension of a broken line $\ell_v$ which is distinct from a path of length $2$. 
The weight $w(v)$ of each vertex $v$ is required to be greater than or equal its degree in $T_r$. We color $v$ black if its weight is strictly greater than its degree. Otherwise, we color $v$ white.
\item $\Delta_v\neq\Delta_{v'}$ if $v\neq v'$ and $\bigcup_{v\in V(T_r)}^{}\Delta_v=\Delta$.
\item Each edge $e$ of $T_r$ is associated to an induced 4--cycle $\Delta_e$ of $\Delta$. If $e\neq e'$, then $\Delta_e\neq\Delta_{e'}$.
\item Two vertices $v_1$ and $v_2$ of $T_r$ are endpoints of the same edge $e$ if and only if $\Delta_{v_1}\cap\Delta_{v_2}=\Delta_e$. Moreover, the induced $4$-cycle $\Delta_e$ bounds a region in both $\field{S}^2-\Delta_{v_1}$ and $\field{S}^2-\Delta_{v_2}$. Also, suspension vertices of $\Delta_{v_1}$ are vertices of $\ell_{v_2}$. Similarly, suspension vertices of $\Delta_{v_2}$ are vertices of $\ell_{v_1}$. Lastly, if $V_1$ and $V_2$ are vertex sets of two components of $T_r$ removed the midpoint of $e$, then $(\bigcup_{v\in V_1}\Delta_v)\cap(\bigcup_{v\in V_2}\Delta_v)=\Gamma_e$. 
\end{enumerate}

\begin{defn}[Visual decomposition trees]
Let $\Delta\subset \field{S}^2$ be a planar flag complex satisfying Standing Assumptions with the $1$-skeleton a non-join $\mathcal{CFS}$ graph. A tree that encodes the structure of $\Delta$ carrying Properties (1), (2), (3), and (4) as above is called a \emph{visual decomposition tree} of $\Delta$.
\end{defn}

\begin{rem}
The existence of a visual decomposition tree for a planar flag complex satisfying Standing Assumptions with the $1$-skeleton a non-join $\mathcal{CFS}$ graph is guaranteed by Construction~\ref{cons:keyidea2}. 
\end{rem}

\begin{figure}
\begin{tikzpicture}[scale=0.25]

\filldraw[fill=blue!40!white, draw=black] (0,0) rectangle (2,2);
\filldraw[fill=red!40!white, draw=black] (0,0) rectangle (-1,2);
\filldraw[fill=red!40!white, draw=black] (0,2) rectangle (2,3);
\filldraw[fill=red!40!white, draw=black] (2,2) rectangle (3,0);
\filldraw[fill=red!40!white, draw=black] (0,0) rectangle (2,-1);
\filldraw[fill=green!40!white, draw=black] (0,0) -- (-1,0) -- (0,-1) -- (0,0);
\filldraw[fill=green!40!white, draw=black] (2,0) -- (3,0) -- (2,-1) -- (2,0);
\filldraw[fill=green!40!white, draw=black] (2,2) -- (3,2) -- (2,3) -- (2,2);
\filldraw[fill=green!40!white, draw=black] (0,2) -- (-1,2) -- (0,3) -- (0,2);

\filldraw[fill=blue!40!white, draw=black] (0+5,0+5) rectangle (2+5,2+5);
\filldraw[fill=red!40!white, draw=black] (0+5,0+5) rectangle (-1+5,2+5);
\filldraw[fill=red!40!white, draw=black] (0+5,2+5) rectangle (2+5,3+5);
\filldraw[fill=red!40!white, draw=black] (2+5,2+5) rectangle (3+5,0+5);
\filldraw[fill=red!40!white, draw=black] (0+5,0+5) rectangle (2+5,-1+5);
\filldraw[fill=green!40!white, draw=black] (0+5,0+5) -- (-1+5,0+5) -- (0+5,-1+5) -- (0+5,0+5);
\filldraw[fill=green!40!white, draw=black] (2+5,0+5) -- (3+5,0+5) -- (2+5,-1+5) -- (2+5,0+5);
\filldraw[fill=green!40!white, draw=black] (2+5,2+5) -- (3+5,2+5) -- (2+5,3+5) -- (2+5,2+5);
\filldraw[fill=green!40!white, draw=black] (0+5,2+5) -- (-1+5,2+5) -- (0+5,3+5) -- (0+5,2+5);

\filldraw[fill=blue!40!white, draw=black] (0+5,0-5) rectangle (2+5,2-5);
\filldraw[fill=red!40!white, draw=black] (0+5,0-5) rectangle (-1+5,2-5);
\filldraw[fill=red!40!white, draw=black] (0+5,2-5) rectangle (2+5,3-5);
\filldraw[fill=red!40!white, draw=black] (2+5,2-5) rectangle (3+5,0-5);
\filldraw[fill=red!40!white, draw=black] (0+5,0-5) rectangle (2+5,-1-5);
\filldraw[fill=green!40!white, draw=black] (0+5,0-5) -- (-1+5,0-5) -- (0+5,-1-5) -- (0+5,0-5);
\filldraw[fill=green!40!white, draw=black] (2+5,0-5) -- (3+5,0-5) -- (2+5,-1-5) -- (2+5,0-5);
\filldraw[fill=green!40!white, draw=black] (2+5,2-5) -- (3+5,2-5) -- (2+5,3-5) -- (2+5,2-5);
\filldraw[fill=green!40!white, draw=black] (0+5,2-5) -- (-1+5,2-5) -- (0+5,3-5) -- (0+5,2-5);

\filldraw[fill=blue!40!white, draw=black] (0-5,0-5) rectangle (2-5,2-5);
\filldraw[fill=red!40!white, draw=black] (0-5,0-5) rectangle (-1-5,2-5);
\filldraw[fill=red!40!white, draw=black] (0-5,2-5) rectangle (2-5,3-5);
\filldraw[fill=red!40!white, draw=black] (2-5,2-5) rectangle (3-5,0-5);
\filldraw[fill=red!40!white, draw=black] (0-5,0-5) rectangle (2-5,-1-5);
\filldraw[fill=green!40!white, draw=black] (0-5,0-5) -- (-1-5,0-5) -- (0-5,-1-5) -- (0-5,0-5);
\filldraw[fill=green!40!white, draw=black] (2-5,0-5) -- (3-5,0-5) -- (2-5,-1-5) -- (2-5,0-5);
\filldraw[fill=green!40!white, draw=black] (2-5,2-5) -- (3-5,2-5) -- (2-5,3-5) -- (2-5,2-5);
\filldraw[fill=green!40!white, draw=black] (0-5,2-5) -- (-1-5,2-5) -- (0-5,3-5) -- (0-5,2-5);

\filldraw[fill=blue!40!white, draw=black] (0-5,0+5) rectangle (2-5,2+5);
\filldraw[fill=red!40!white, draw=black] (0-5,0+5) rectangle (-1-5,2+5);
\filldraw[fill=red!40!white, draw=black] (0-5,2+5) rectangle (2-5,3+5);
\filldraw[fill=red!40!white, draw=black] (2-5,2+5) rectangle (3-5,0+5);
\filldraw[fill=red!40!white, draw=black] (0-5,0+5) rectangle (2-5,-1+5);
\filldraw[fill=green!40!white, draw=black] (0-5,0+5) -- (-1-5,0+5) -- (0-5,-1+5) -- (0-5,0+5);
\filldraw[fill=green!40!white, draw=black] (2-5,0+5) -- (3-5,0+5) -- (2-5,-1+5) -- (2-5,0+5);
\filldraw[fill=green!40!white, draw=black] (2-5,2+5) -- (3-5,2+5) -- (2-5,3+5) -- (2-5,2+5);
\filldraw[fill=green!40!white, draw=black] (0-5,2+5) -- (-1-5,2+5) -- (0-5,3+5) -- (0-5,2+5);

\filldraw[fill=red!40!white, draw=black] (3,2) -- (5,4) -- (4,5) -- (2,3) -- (3,2);
\filldraw[fill=red!40!white, draw=black] (-1,2) -- (-3,4) -- (-2,5) -- (0,3) -- (-1,2);
\filldraw[fill=red!40!white, draw=black] (-1,0) -- (-3,-2) -- (-2,-3) -- (0,-1) -- (-1,0);
\filldraw[fill=red!40!white, draw=black] (3,0) -- (5,-2) -- (4,-3) -- (2,-1) -- (3,0);

\filldraw[fill=red!40!white, draw=black] (3,2-10) -- (5,4-10) -- (4,5-10) -- (2,3-10) -- (3,2-10);
\filldraw[fill=red!40!white, draw=black] (-1,2-10) -- (-3,4-10) -- (-2,5-10) -- (0,3-10) -- (-1,2-10);
\filldraw[fill=red!40!white, draw=black] (-1,0+10) -- (-3,-2+10) -- (-2,-3+10) -- (0,-1+10) -- (-1,0+10);
\filldraw[fill=red!40!white, draw=black] (3,0+10) -- (5,-2+10) -- (4,-3+10) -- (2,-1+10) -- (3,0+10);

\filldraw[fill=red!40!white, draw=black] (3-10,2) -- (5-10,4) -- (4-10,5) -- (2-10,3) -- (3-10,2);
\filldraw[fill=red!40!white, draw=black] (-1+10,2) -- (-3+10,4) -- (-2+10,5) -- (0+10,3) -- (-1+10,2);
\filldraw[fill=red!40!white, draw=black] (-1+10,0) -- (-3+10,-2) -- (-2+10,-3) -- (0+10,-1) -- (-1+10,0);
\filldraw[fill=red!40!white, draw=black] (3-10,0) -- (5-10,-2) -- (4-10,-3) -- (2-10,-1) -- (3-10,0);

\filldraw[fill=red!40!white, draw=black] (3-10,2-10) -- (5-10,4-10) -- (4-10,5-10) -- (2-10,3-10) -- (3-10,2-10);
\filldraw[fill=red!40!white, draw=black] (3+5,2+5) -- (5+5,4+5) -- (4+5,5+5) -- (2+5,3+5) -- (3+5,2+5);
\filldraw[fill=red!40!white, draw=black] (3+5,0-5) -- (5+5,-2-5) -- (4+5,-3-5) -- (2+5,-1-5) -- (3+5,0-5);
\filldraw[fill=red!40!white, draw=black] (3-10,0+10) -- (5-10,-2+10) -- (4-10,-3+10) -- (2-10,-1+10) -- (3-10,0+10);

\node at (1,-10) {$F(\Sigma_v$)};


\filldraw[fill=blue!40!white, draw=red] (-20,0) rectangle (-18,2);
\filldraw[fill=blue!40!white, draw=red] (-20+3,0+3) rectangle (-18+3,2+3);
\filldraw[fill=blue!40!white, draw=red] (-20-3,0+3) rectangle (-18-3,2+3);
\filldraw[fill=blue!40!white, draw=red] (-20-3,0-3) rectangle (-18-3,2-3);
\filldraw[fill=blue!40!white, draw=red] (-20+3,0-3) rectangle (-18+3,2-3);

\draw[red] (-20,0)node[circle,fill,inner sep=1.0pt, color=green](1){}--(-21,-1); \draw[red] (-18,0)node[circle,fill,inner sep=1.0pt, color=green](1){}--(-17,-1); \draw[red] (-20,2)node[circle,fill,inner sep=1.0pt, color=green](1){}--(-21,3); \draw[red] (-18,2)node[circle,fill,inner sep=1.0pt, color=green](1){}--(-17,3);
\draw[red] (-23,-1)node[circle,fill,inner sep=1.0pt, color=green](1){}--(-24,0); \draw[red] (-23,-3)node[circle,fill,inner sep=1.0pt, color=green](1){}--(-24,-4); \draw[red] (-21,-3)node[circle,fill,inner sep=1.0pt, color=green](1){}--(-20,-4);
\draw[red] (-15,-1)node[circle,fill,inner sep=1.0pt, color=green](1){}--(-14,0); \draw[red] (-15,-3)node[circle,fill,inner sep=1.0pt, color=green](1){}--(-14,-4); \draw[red] (-17,-3)node[circle,fill,inner sep=1.0pt, color=green](1){}--(-18,-4);
\draw[red] (-17,5)node[circle,fill,inner sep=1.0pt, color=green](1){}--(-18,6); \draw[red] (-15,5)node[circle,fill,inner sep=1.0pt, color=green](1){}--(-14,6); \draw[red] (-15,3)node[circle,fill,inner sep=1.0pt, color=green](1){}--(-14,2);
\draw[red] (-23,3)node[circle,fill,inner sep=1.0pt, color=green](1){}--(-24,2); \draw[red] (-23,5)node[circle,fill,inner sep=1.0pt, color=green](1){}--(-24,6); \draw[red] (-21,5)node[circle,fill,inner sep=1.0pt, color=green](1){}--(-20,6);

\node at (-19,-6) {$\Sigma_v=\Sigma_{\ell_v}$};


\draw[blue] (-35,1)node[circle,fill,inner sep=1.0pt, color=red](1){}--(-31,1)node[circle,fill,inner sep=1.0pt, color=red](1){};
\draw (-39,1)node[circle,fill,inner sep=1.0pt, color=red](1){};

\node at (-35,-2) {$\ell_v$};

\end{tikzpicture}

\caption{A broken line $\ell_v$, the associated Davis complex $\Sigma_v$, and the fattening $F(\Sigma_v)$ of $\Sigma_v$}
\label{coolcool}

\end{figure}
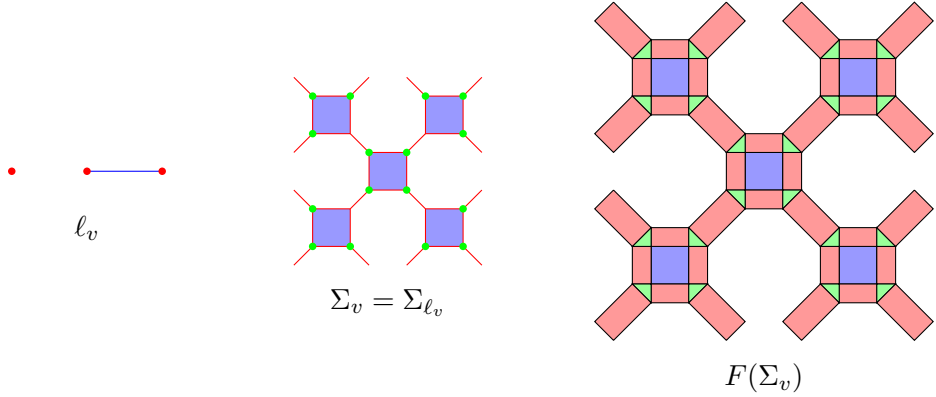

Let $\Delta\subset \field{S}^2$ be a planar flag complex satisfying Standing Assumptions with the $1$-skeleton a non-join $\mathcal{CFS}$ graph. Let $T_r$ be a visual decomposition tree of $\Delta$. Since $\Delta$ is planar, $G_{\Delta}$ is virtually a $3$--manifold group by Theorem~\ref{mani}. However, for the purpose of obtaining a quasi-isometry classification we will construct explicitly a $3$--manifold $Y$ on which $G_{\Delta}$ acts properly and cocompactly. We note that the construction of the manifold $Y$ is associated to the graph $T_{r}$. Therefore, we can import the work of Behrstock-Neumann \cite{MR2376814} to prove Theorem~\ref{thm:graphmanifold}.

\begin{cons}
\label{cons:constructionmanifold}
We now construct a 3-manifold $Y$ on which the right-angled Coxeter group $G_\Delta$ acts properly and cocompactly. For each vertex $v$ of $T_r$, the complex $\Delta_v$ is a suspension of a broken line $\ell_v$ of $\Delta$. We now let $b$ and $c$ be suspension vertices and let $\Sigma_v$ be the Davis complex associated to the broken line $\ell_v$. We fatten the $1$-skeleton of $\Sigma_v$ to obtain a universal cover of a hyperbolic surface with boundary as follows:

Let $n$ be the number vertices of  $\ell_v$. We replace each vertex of the $1$-skeleton $\Sigma^{(1)}_v$ by a regular $n$--gon with sides labelled by vertices of $\ell_v$. We assume that two edges of such an $n$--gon labeled by $a_1$ and $a_2$ in $V(\ell_v)$ are adjacent if and only if either $a_1$ and $a_2$ are vertices of an edge of $\ell_v$ or the set $\{a_1, a_2,b,c\}$ forms an induced $4$--cycle of $\Delta_v$ that bounds a region in $\field{S}^2-\Delta_v$. We also assume the length side of the $n$--gon is $1/2$. We replace each edge $E$ labelled by $a_i$ by a strip $E\times[-1/4,1/4]$. We label each side of length $1$ of the strip $E\times[-1/4,1/4]$ by $a_i$ and we identify the edge $E$ to $E\times\{0\}$ of the strip. If $u$ is an endpoint of the edge $E$ of $\Sigma^{(1)}_v$, then the edge $\{u\}\times[-1/4,1/4]$ is identified to the side labelled by $a_i$ of the $n$--gon that replaces $u$. 

We observe that for each pair of adjacent vertices $a_1$ and $a_2$ of $\ell_v$ we have an induced $4$--cycle in $\Sigma^{(1)}_v$ with one pair of opposite sides labeled by $a_1$ and the other pair of opposite sides labeled by $a_2$. We note that this $4$--cycle bounds a $2$--cell in $\Sigma_v$. In the ``fattening'' of $\Sigma^{(1)}_v$ constructed above we also have such $4$--cycles and we also fill each with a $2$--cell as in $\Sigma_v$. We denote $F(\Sigma_v)$ to be the resulting space (See Figure~\ref{coolcool}). Clearly, $G_{\ell_v}$ acts properly and cocompactly on the space $F(\Sigma_v)$. Additionally, $F(\Sigma_v)$ is a simply connected surface with boundary and each boundary component of $F(\Sigma_v)$ is an infinite concatenation of edges labelled by $a_{1}$ and $a_2$ where $a_1$ and $a_2$ are vertices of $\ell_v$ such that the set $\{a_1, a_2,b,c\}$ forms an induced $4$--cycle that bounds a region in $\field{S}^2-\Delta_v$. We denote such a boundary by $\alpha_{a_1,a_2}$.

The group $G_{\{b,c\}}$ acts on the line $\alpha$ that is a concatenation of edges labelled by $b$ and $c$ by edge reflections. Let $P_v=F(\Sigma_v)\times \alpha$ and equip $P_v$ with the product metric. Then, the group $G_{\Delta_v}$ acts properly and cocompactly on $P_v$ in the obvious way. The space $P_v$ is the universal cover of the trivial circle bundle of a hyperbolic surface with nonempty boundary.

Moreover, for each pair of vertices $a_1$ and $a_2$ of $\ell_v$ such that the set $\{a_1, a_2,b,c\}$ forms an induced $4$--cycle that bounds a region in $\field{S}^2-\Delta_v$ the right-angled Coxeter groups generated by $\{a_{1},a_2,b,c\}$ acts on the boundary $\alpha_{a_1,a_2}\times \alpha$ in a way analogous to its action on the Davis complex. We label this plane by $\{a_{1},a_2,b,c\}$. 

If $v_1$ and $v_2$ are two adjacent vertices in $T_r$, then the pair of suspension vertices $(a_1,a_2)$ of $\Delta_{v_1}$ is a pair of vertices of $\ell_{v_2}$ and the pair of suspension vertices $(b_1,b_2)$ of $\Delta_{v_2}$ is a pair of of vertices of $\ell_{v_1}$. Moreover, the set $\{a_1, a_2,b_1,b_2\}$ forms an induced $4$--cycle that bounds a region in both $\field{S}^2-\Delta_{v_1}$ and $\field{S}^2-\Delta_{v_2}$. Therefore, the spaces $P_{v_1}$ and $P_{v_2}$ have two Euclidean planes which are labeled by $\{a_1,a_2,b_1,b_2\}$ as constructed above. Thus, using the Bass-Serre tree $\tilde{T}_r$ of the decomposition of $G_\Delta$ as tree $T_r$ of subgroups we can form a 3-manifold $Y$ by gluing copies of $P_v$ appropriately and we obtain a proper, cocompact action of $G_\Delta$ on $Y$.
\end{cons}

We are ready to prove the quasi-isometry classification theorem. The proof is identical with the proof of (3) in Theorem~1.1 \cite{NT}.

\begin{proof}[Proof of Theorem~\ref{thm:graphmanifold}]

We can color vertices of the Bass-Serre tree $\tilde{T}_r$ so that the quotient map $q: \tilde{T}_r\to T_r$ preserves the vertex coloring. Moreover, the map $q: \tilde{T}_r\to T_r$ is a week covering. Therefore, the trees $\tilde{T}_r$ and $T_r$ are bisimilar. (See Definitions \ref{weaklycover} and \ref{bisimilarity1} for the definitions of weak covering and bisimilarity.) We observe that a vertex of $\tilde{T}_r$ is colored black if and only if the corresponding Seifert manifold contains a component of the boundary of $Y$. Using the proof of Theorem 3.2 in \cite{MR2376814}, we obtain the proof of Theorem~\ref{thm:graphmanifold}.
\end{proof}

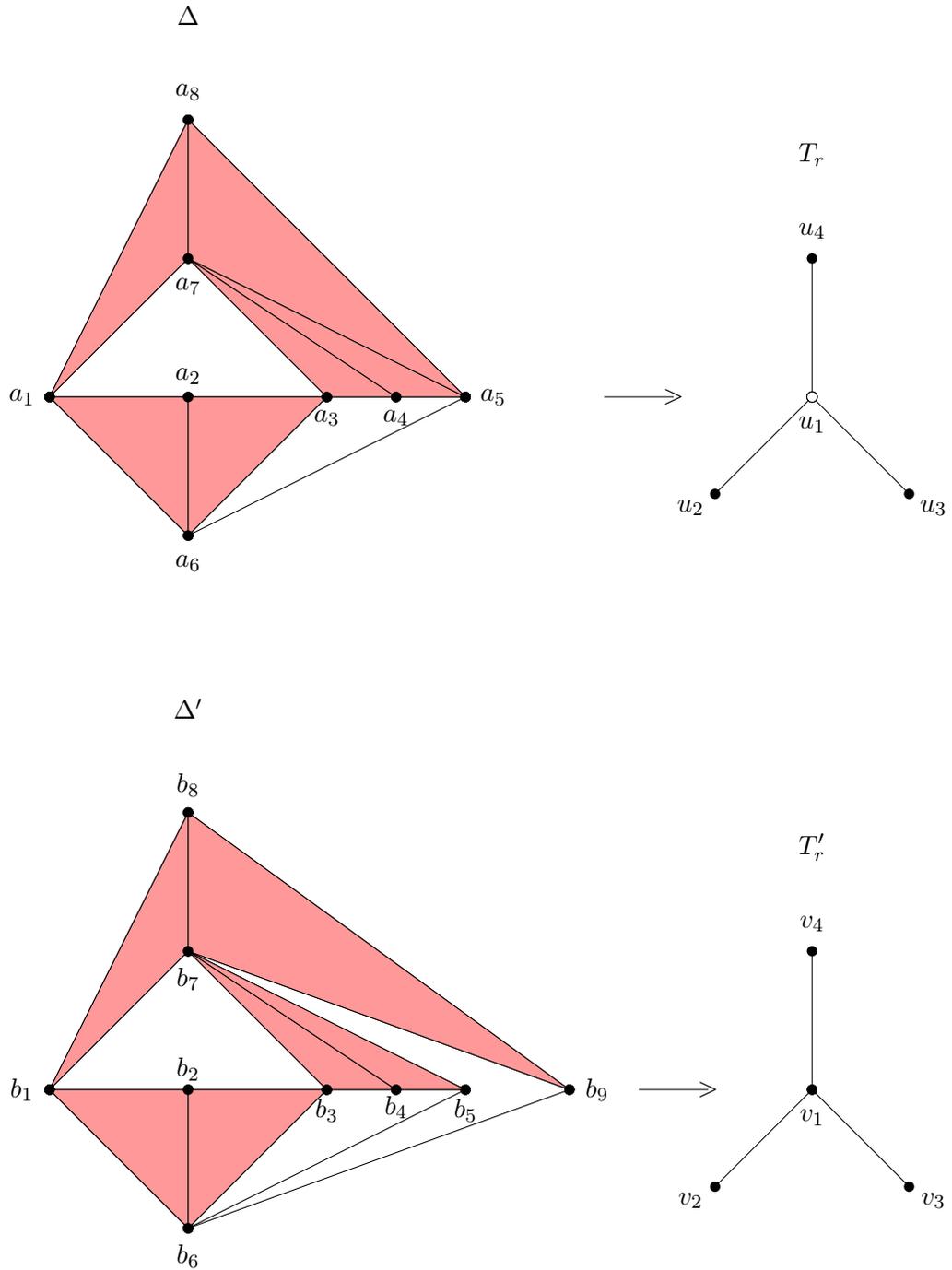
\begin{figure}
\begin{tikzpicture}[scale=1.0]

\draw (0,0) node[circle,fill,inner sep=1.5pt, color=black](1){} -- (2,0) node[circle,fill,inner sep=1.5pt, color=black](1){}-- (4,0) node[circle,fill,inner sep=1.5pt, color=black](1){}-- (5,0) node[circle,fill,inner sep=1.5pt, color=black](1){} -- (6,0) node[circle,fill,inner sep=1.5pt, color=black](1){};

\draw (2,2) node[circle,fill,inner sep=1.5pt, color=black](1){} -- (0,0) node[circle,fill,inner sep=1.5pt, color=black](1){};
\draw (2,2) node[circle,fill,inner sep=1.5pt, color=black](1){} -- (4,0) node[circle,fill,inner sep=1.5pt, color=black](1){};
\draw (2,2) node[circle,fill,inner sep=1.5pt, color=black](1){} -- (6,0) node[circle,fill,inner sep=1.5pt, color=black](1){};

\draw (2,-2) node[circle,fill,inner sep=1.5pt, color=black](1){} -- (0,0) node[circle,fill,inner sep=1.5pt, color=black](1){};
\draw (2,-2) node[circle,fill,inner sep=1.5pt, color=black](1){} -- (4,0) node[circle,fill,inner sep=1.5pt, color=black](1){};
\draw (2,-2) node[circle,fill,inner sep=1.5pt, color=black](1){} -- (6,0) node[circle,fill,inner sep=1.5pt, color=black](1){};

\draw (0,0) node[circle,fill,inner sep=1.5pt, color=black](1){} -- (2,4) node[circle,fill,inner sep=1.5pt, color=black](1){} -- (6,0) node[circle,fill,inner sep=1.5pt, color=black](1){};

\node at (-.4,0) {$a_1$};\node at (2,0.3) {$a_2$};\node at (4,-0.3) {$a_3$};\node at (5,-0.3) {$a_4$};\node at (6.4,0) {$a_5$}; \node at (2,-2.4) {$a_6$};\node at (2,1.6) {$a_7$};\node at (2,4.4) {$a_8$}; \draw (8,0)--(9,0);\node at (9,0) {$>$}; 

\node at (2.0,5.5) {$\Delta$};

\filldraw[fill=red!40!white, draw=black] (0,0) node[circle,fill,inner sep=1.5pt, color=black](1){} -- (2,-2) node[circle,fill,inner sep=1.5pt, color=black](1){} -- (2,0)node[circle,fill,inner sep=1.5pt, color=black](1){} -- (0,0)node[circle,fill,inner sep=1.5pt, color=black](1){};
\filldraw[fill=red!40!white, draw=black] (4,0) node[circle,fill,inner sep=1.5pt, color=black](1){} -- (2,-2) node[circle,fill,inner sep=1.5pt, color=black](1){} -- (2,0)node[circle,fill,inner sep=1.5pt, color=black](1){} -- (4,0)node[circle,fill,inner sep=1.5pt, color=black](1){};

\filldraw[fill=red!40!white, draw=black] (4,0) node[circle,fill,inner sep=1.5pt, color=black](1){} -- (5,0) node[circle,fill,inner sep=1.5pt, color=black](1){} -- (2,2)node[circle,fill,inner sep=1.5pt, color=black](1){} -- (4,0)node[circle,fill,inner sep=1.5pt, color=black](1){};
\filldraw[fill=red!40!white, draw=black] (5,0) node[circle,fill,inner sep=1.5pt, color=black](1){} -- (6,0) node[circle,fill,inner sep=1.5pt, color=black](1){} -- (2,2)node[circle,fill,inner sep=1.5pt, color=black](1){} -- (5,0)node[circle,fill,inner sep=1.5pt, color=black](1){};

\filldraw[fill=red!40!white, draw=black] (0,0) node[circle,fill,inner sep=1.5pt, color=black](1){} -- (2,2) node[circle,fill,inner sep=1.5pt, color=black](1){} -- (2,4)node[circle,fill,inner sep=1.5pt, color=black](1){} -- (0,0)node[circle,fill,inner sep=1.5pt, color=black](1){};
\filldraw[fill=red!40!white, draw=black] (6,0) node[circle,fill,inner sep=1.5pt, color=black](1){} -- (2,2) node[circle,fill,inner sep=1.5pt, color=black](1){} -- (2,4)node[circle,fill,inner sep=1.5pt, color=black](1){} -- (6,0)node[circle,fill,inner sep=1.5pt, color=black](1){};

\draw (11,0) node[circle,draw=black,inner sep=1.5pt, fill=white](1){} -- (11,2) node[circle,fill,inner sep=1.5pt, color=black](1){};
\draw (11,0) node[circle,draw=black,inner sep=1.5pt, fill=white](1){} -- (9.6,-1.4) node[circle,fill,inner sep=1.5pt, color=black](1){};
\draw (11,0) node[circle,draw=black,inner sep=1.5pt, fill=white](1){} -- (12.4,-1.4) node[circle,fill,inner sep=1.5pt, color=black](1){};

\node at (11,-0.4) {$u_1$};\node at (9.25,-1.6) {$u_2$};\node at (12.75,-1.6) {$u_3$};\node at (11,2.4) {$u_4$};

\node at (11,3.5) {$T_r$};

\draw (0,-10) node[circle,fill,inner sep=1.5pt, color=black](1){} -- (2,-10) node[circle,fill,inner sep=1.5pt, color=black](1){}-- (4,-10) node[circle,fill,inner sep=1.5pt, color=black](1){}-- (5,-10) node[circle,fill,inner sep=1.5pt, color=black](1){} -- (6,-10) node[circle,fill,inner sep=1.5pt, color=black](1){};

\draw (2,-8) node[circle,fill,inner sep=1.5pt, color=black](1){} -- (0,-10) node[circle,fill,inner sep=1.5pt, color=black](1){};
\draw (2,-8) node[circle,fill,inner sep=1.5pt, color=black](1){} -- (4,-10) node[circle,fill,inner sep=1.5pt, color=black](1){};
\draw (2,-8) node[circle,fill,inner sep=1.5pt, color=black](1){} -- (6,-10) node[circle,fill,inner sep=1.5pt, color=black](1){};
\draw (2,-8) node[circle,fill,inner sep=1.5pt, color=black](1){} -- (7.5,-10) node[circle,fill,inner sep=1.5pt, color=black](1){};

\draw (2,-12) node[circle,fill,inner sep=1.5pt, color=black](1){} -- (0,-10) node[circle,fill,inner sep=1.5pt, color=black](1){};
\draw (2,-12) node[circle,fill,inner sep=1.5pt, color=black](1){} -- (4,-10) node[circle,fill,inner sep=1.5pt, color=black](1){};
\draw (2,-12) node[circle,fill,inner sep=1.5pt, color=black](1){} -- (6,-10) node[circle,fill,inner sep=1.5pt, color=black](1){};
\draw (2,-12) node[circle,fill,inner sep=1.5pt, color=black](1){} -- (7.5,-10) node[circle,fill,inner sep=1.5pt, color=black](1){};

\draw (0,-10) node[circle,fill,inner sep=1.5pt, color=black](1){} -- (2,-6) node[circle,fill,inner sep=1.5pt, color=black](1){} -- (7.5,-10) node[circle,fill,inner sep=1.5pt, color=black](1){};

\node at (-.4,-10) {$b_1$};\node at (2,-9.7) {$b_2$};\node at (4,-10.3) {$b_3$};\node at (5,-10.25) {$b_4$};\node at (6,-10.3) {$b_5$}; \node at (2,-12.4) {$b_6$};\node at (2,-8.4) {$b_7$};\node at (2,-5.6) {$b_8$}; \node at (7.9,-10) {$b_9$};\draw (8.5,-10)--(9.5,-10);\node at (9.5,-10) {$>$}; 

\node at (2.0,-4.5) {$\Delta'$};

\filldraw[fill=red!40!white, draw=black] (0,-10) node[circle,fill,inner sep=1.5pt, color=black](1){} -- (2,-2-10) node[circle,fill,inner sep=1.5pt, color=black](1){} -- (2,0-10)node[circle,fill,inner sep=1.5pt, color=black](1){} -- (0,0-10)node[circle,fill,inner sep=1.5pt, color=black](1){};
\filldraw[fill=red!40!white, draw=black] (4,0-10) node[circle,fill,inner sep=1.5pt, color=black](1){} -- (2,-2-10) node[circle,fill,inner sep=1.5pt, color=black](1){} -- (2,0-10)node[circle,fill,inner sep=1.5pt, color=black](1){} -- (4,0-10)node[circle,fill,inner sep=1.5pt, color=black](1){};

\filldraw[fill=red!40!white, draw=black] (4,0-10) node[circle,fill,inner sep=1.5pt, color=black](1){} -- (5,0-10) node[circle,fill,inner sep=1.5pt, color=black](1){} -- (2,2-10)node[circle,fill,inner sep=1.5pt, color=black](1){} -- (4,0-10)node[circle,fill,inner sep=1.5pt, color=black](1){};
\filldraw[fill=red!40!white, draw=black] (5,0-10) node[circle,fill,inner sep=1.5pt, color=black](1){} -- (6,0-10) node[circle,fill,inner sep=1.5pt, color=black](1){} -- (2,2-10)node[circle,fill,inner sep=1.5pt, color=black](1){} -- (5,0-10)node[circle,fill,inner sep=1.5pt, color=black](1){};

\filldraw[fill=red!40!white, draw=black] (0,0-10) node[circle,fill,inner sep=1.5pt, color=black](1){} -- (2,2-10) node[circle,fill,inner sep=1.5pt, color=black](1){} -- (2,4-10)node[circle,fill,inner sep=1.5pt, color=black](1){} -- (0,0-10)node[circle,fill,inner sep=1.5pt, color=black](1){};
\filldraw[fill=red!40!white, draw=black] (7.5,0-10) node[circle,fill,inner sep=1.5pt, color=black](1){} -- (2,2-10) node[circle,fill,inner sep=1.5pt, color=black](1){} -- (2,4-10)node[circle,fill,inner sep=1.5pt, color=black](1){} -- (7.5,0-10)node[circle,fill,inner sep=1.5pt, color=black](1){};

\draw (11,-10) node[circle,fill,inner sep=1.5pt, color=black](1){} -- (11,-8) node[circle,fill,inner sep=1.5pt, color=black](1){};
\draw (11,-10) node[circle,fill,inner sep=1.5pt, color=black](1){} -- (9.6,-11.4) node[circle,fill,inner sep=1.5pt, color=black](1){};
\draw (11,-10) node[circle,fill,inner sep=1.5pt, color=black](1){} -- (12.4,-11.4) node[circle,fill,inner sep=1.5pt, color=black](1){};

\node at (11,-10.4) {$v_1$};\node at (9.25,-11.6) {$v_2$};\node at (12.75,-11.6) {$v_3$};\node at (11,-7.6) {$v_4$};

\node at (11,-6.5) {$T_r'$};

\end{tikzpicture}

\caption{The groups $G_{\Delta}$ and $G_{\Delta'}$ are not quasi-isometric because the two corresponding decomposition trees $T_r$ and $T_r'$ are not bisimilar.}
\label{afifthf}
\end{figure}

\begin{exmp}
\label{niceexample}
Let $\Delta$ and $\Delta'$ be flag complexes in Figure \ref{afifthf}. A visual decomposition tree $T_r$ of $\Delta$ is shown in the same figure with the following information. The complex $\Delta_{u_1}$ is the suspension of the broken line with three vertices $a_1$, $a_3$, and $a_5$ with two suspension vertices $a_6$ and $a_7$. The complex $\Delta_{u_2}$ is the suspension of the broken line with three vertices $a_2$, $a_6$, and $a_7$ with two suspension vertices $a_1$ and $a_3$. The complex $\Delta_{u_3}$ is the suspension of the broken line with three vertices $a_4$, $a_6$, and $a_7$ with two suspension vertices $a_3$ and $a_5$. The complex $\Delta_{u_4}$ is the suspension of the broken line with three vertices $a_6$, $a_7$, and $a_8$ with two suspension vertices $a_1$ and $a_5$. We note, that the vertices $u_2$, $u_3$, and $u_4$ all have weight $2$ and the vertex $u_1$ has weight $3$. By comparing their degrees, the vertices $u_2$, $u_3$, and $u_4$ are colored black while $u_1$ is colored white.

Similarly, a visual decomposition tree $T_r'$ of $\Delta'$ is also shown in Figure \ref{afifthf} with the following information. The complex $\Delta_{v_1}$ is the suspension of the broken line with four vertices $b_1$, $b_3$, $b_5$, and $b_9$ with two suspension vertices $b_6$ and $b_7$. The complex $\Delta_{v_2}$ is the suspension of the broken line with three vertices $b_2$, $b_6$, and $b_7$ with two suspension vertices $b_1$ and $b_3$. The complex $\Delta_{v_3}$ is the suspension of the broken line with three vertices $b_4$, $b_6$, and $b_7$ with two suspension vertices $b_3$ and $b_5$. The complex $\Delta_{v_4}$ is the suspension of the broken line with three vertices $b_6$, $b_7$, and $b_8$ with two suspension vertices $b_1$ and $b_9$. We observe that the vertices $v_2$, $v_3$, and $v_4$ all have weight $2$ and the vertex $v_1$ has weight $4$. By comparing their degrees we see that all vertices of $T_r'$ are colored black. Thus, the visual decomposition trees $T_r$ and $T_r'$ are not bisimilar even though they are isomorphic if we ignore the vertex colors. Therefore, the groups $G_\Gamma$ and $G_{\Gamma'}$ are not quasi-isometric.
\end{exmp} 

\bibliographystyle{alpha}
\bibliography{Tran}
\end{document}